\documentclass{article}

\usepackage[latin1]{inputenc}
\usepackage{a4wide}
\usepackage[]{latexsym}
\usepackage{amsmath, amsthm, amsfonts}
\usepackage{verbatim}

\usepackage{amssymb}
\usepackage{amscd}
\usepackage{longtable}
\usepackage{array}
\usepackage[all]{xy}
\usepackage{a4}

\newtheorem{defn0}{Definition}[section]
\newtheorem{prop0}[defn0]{Proposition}
\newtheorem{thm0}[defn0]{Theorem}
\newtheorem{lemma0}[defn0]{Lemma}
\newtheorem{claim0}[defn0]{Claim}
\newtheorem{corollary0}[defn0]{Corollary}
\newtheorem{example0}[defn0]{Example}
\newtheorem{remark0}[defn0]{Remark}
\newtheorem{assumption0}[defn0]{Assumption}
\newtheorem{conjecture0}[defn0]{Conjecture}
\newtheorem{notation0}[defn0]{Notation}
\newtheorem{question0}[defn0]{Question}

\newenvironment{definition}{\begin{defn0}\rm}{\end{defn0}}
\newenvironment{proposition}{\begin{prop0}}{\end{prop0}}
\newenvironment{theorem}{\begin{thm0}}{\end{thm0}}
\newenvironment{lemma}{\begin{lemma0}}{\end{lemma0}}

\newenvironment{corollary}{\begin{corollary0}}{\end{corollary0}}

\newenvironment{remark}{\begin{remark0}\rm}{\end{remark0}}

\newcommand{\Fr}{{\mathrm {Fr}}}
\newcommand{\disc}{{\mathrm {disc }}}

\newcommand{\E}{\mathcal{E}}
\newcommand{\Pic}{\mathrm{Pic}}

\newcommand{\Trace}{{\mathrm{Tr}}}
\newcommand{\M}{\mathrm{M}}

\newcommand{\Br}{\mathrm{Br}}

\newcommand{\Frob}{\mathrm{F}}
\newcommand{\Vi}{\mathrm{V}}

\newcommand{\End}{{\mathrm{End}}}

\newcommand{\Ann}{{\mathrm {Ann}}}

\newcommand{\Z}{{\mathbb Z}}
\newcommand{\Q}{{\mathbb Q}}
\newcommand{\C}{{\mathbb C}}
\newcommand{\R}{{\mathbb R}}
\newcommand{\F}{{\mathbb F}}

\newcommand{\PP}{{\mathbb P}}

\newcommand{\cS}{{\mathcal S}}

\newcommand{\cX}{{\mathcal X}}
\newcommand{\cM}{{\mathcal M}}
\newcommand{\dN}{{\mathcal N}}

\newcommand{\cE}{{\mathcal E}}

\newcommand{\cO}{{\mathcal O}}

\newcommand{\dP}{{\mathfrak P}}
\newcommand{\dQ}{{\mathfrak Q}}

\newcommand{\id}{{\mathrm {Id}}}
\newcommand{\Spec}{{\mathrm {Spec}}}

\newcommand{\ra}{{\rightarrow}}
\newcommand{\lra}{\longrightarrow}
\newcommand{\hra}{\hookrightarrow}
\newcommand{\Hom}{{\mathrm {Hom}}}

\newcommand{\car}{{\mathrm{char}}}

\newcommand{\CM}{\operatorname{CM}}

\title{Ribet bimodules and \\ the specialization of Heegner points}
\author{Santiago Molina\\
    \small Matemàtica aplicada IV\\
  \small Universitat Politècnica de Catalunya
}
\begin{document}
\maketitle

\section{Introduction}

Let $B$ be a quaternion algebra over the field $\Q $ of rational numbers, and let $D\geq 1$ denote its reduced discriminant.
Let $N\geq 1$ be a positive integer, coprime to $D$, let $\E(D,N)$ stand for the set of oriented Eichler orders of level $N$ in $B$  (cf.\,\S \ref{QOOE} for precise definitions) and finally let $\Pic(D,N)$ denote the
set of isomorphism classes of such orders.

For any given $\cO \in \E(D,N)$, it is well known that $\Pic(D,N)$ is in bijective correspondence with the set of
classes of projective left $\cO $-ideals in $B$ up to principal ideals.

If $B$ is indefinite, i.\,e.,\,$D$ is the (square-free) product of an {\em even} number of primes, then $\Pic(D,N)$ is trivial for any $N$; if $B$ is definite, $\Pic(D,N)$ is a finite set whose cardinality $h(D, N)$ is often referred to as the class number of $\cO$.

Let $K/\Q$ be an imaginary  quadratic number field and let $R$ be an order of $K$ of conductor $c$. For any $\cO \in \E(D,N)$ let $\CM_{\cO}(R)$ denote the set of optimal embeddings $\varphi: R\hookrightarrow \cO$, up to conjugation by $\cO^{\times}$. Define $$\CM_{D,N}(R) = \sqcup_{\cO} \CM_{\cO}(R),$$ where
$\cO\in\E(D,N)$ runs over a set of representatives of $\Pic(D,N)$. Write
$$
\pi: \CM_{D,N}(R) \lra \Pic(D,N)
$$
for the natural forgetful projection which maps a conjugacy class of optimal embeddings $\varphi: R\hra \cO$ to the isomorphism class of its target $\cO$.

Let $X_0(D,N)/\Q $ denote Shimura's canonical model of the coarse moduli space of abelian surfaces with multiplication by an Eichler order $\cO$ of level $N$ in $B$. As it is well known, 
associated to each $R$ as above there is a (possibly empty) set of algebraic points $\CM(R)\subset X_0(D,N)(K^{ab})$, which is in natural one-to-one correspondence with $\CM_{\cO}(R)$. These points are called {\em special}, {\em CM} or {\em Heegner} points in the literature. Below we shall recall natural actions of the class group $\Pic(R)$ and of the Atkin-Lehner group $W(D,N)$ on both sets $\CM_{\cO}(R)$ and $\CM(R)$ which are intertwined by this correspondence. In the case of $\Pic(R)$ this is the statement of Shimura's reciprocity law.

Let now $p$ be a prime, $p\nmid c$. It follows from the work of various people, including Deligne-Rapoport, Buzzard, Morita, Cerednik and Drinfeld, that $X_0(D,N)$ has good reduction at $p$ if and only if $p\nmid D N$. There exists a proper integral model $\mathcal X_0(D,N)$ of $X_0(D,N)$ over $\mathrm{Spec}(\Z)$, smooth over $\Z[\frac{1}{DN}]$, which suitably extends the moduli interpretation to arbitrary base schemes (cf.\,\cite{BC}, \cite{Mo}). When $p$ is fixed in the context and no confusion can arise, we shall write $\tilde{X}_0(D,N)$ for its special fiber at $p$.
It is known 
that each of the sets $S$ of
\begin{itemize}
\item Singular points of $\tilde{X}_0(D,N)$ for $p\mid \mid DN$,

\item Irreducible components of $\tilde{X}_0(D,N)$ for $p\mid \mid DN$,

\item Supersingular points of $\tilde{X}_0(D,N)$ for $p\nmid DN$
\end{itemize}
are in one-to-one correspondence with either one or two copies of $\Pic(d,n)$ for certain $d$ and $n$ (cf.\,\S \ref{CD}, \ref{ssgood}, \ref{DR} for precise statements).

In each of these three cases, we show that under appropriate behavior of $p$ in $R\subset K$ (that we make explicit in \S \ref{CD}, \ref{ssgood}, \ref{DR}, respectively -see also the table below), Heegner points $P\in \CM(R)$ specialize to elements of $S$ (with the obvious meaning in each case).

This yields a map which, composed with the previously mentioned identifications $\CM_{D,N}(R) = \CM(R)\subset X_0(D,N)(K^{ab})$ and $S=\bigsqcup_{i=1}^{t} \Pic(d,n)$ ($t=1$ or $2$), takes the form

\begin{equation}\label{map}
\CM_{D,N}(R) \lra \bigsqcup_{i=1}^{t} \Pic(d,n).
\end{equation}

Our first observation is that, although the construction of this map is of geometric nature, both the source and the target are pure algebraic objects. Hence, one could ask whether there is a pure algebraic description of the arrow itself.

One of the aims of this note is to exploit Ribet's theory of bimodules in order to show that this is indeed the case, and that in fact  \eqref{map} can be refined as follows: there exists a map $\Phi: \CM_{D,N}(R)\lra \bigsqcup_{i=1}^{t} \CM_{d,n}(R)$ which is equivariant for the actions of $\Pic(R)$ and of the Atkin-Lehner groups and makes the diagram

\begin{equation}\label{map2}
\xymatrix{
\CM_{D,N}(R)\ar[dr] \ar[r]^\Phi & \bigsqcup_{i=1}^{t}\CM_{d,n}(R)\ar[d]^\pi \\
 &  \bigsqcup_{i=1}^{t} \Pic(d,n)
}
\end{equation}
commutative (see \eqref{Phi1}, \eqref{Phi2}, \eqref{Phi3}, \eqref{Phi4} and \eqref{Phi5}). Moreover, we show that when $N$ is square free, the maps $\Phi$ are bijections.

As we explain in this article, the most natural construction of the map $\Phi$ is again geometrical, but can also be recovered in pure algebraic terms.

The main bulk of the article is devoted to the case $p\mid D$; the other cases are comparatively much simpler and probably already known, at least partially.
For the convenience of the reader, we summarize the situation for $p\mid \mid DN$ in the following table.

\vspace{3mm}
\noindent
\begin{tabular}{|c|c|c|c|c|}
\hline
$P\in \CM(R)$ &  & Condition on $p$  &  $S$=Irreducible components & $\Phi$  \\
\hline
$\widetilde P$ & $p\mid D$ & $p$ \small{inert in} $K$ & $\small{\bigsqcup_{i=1}^2\Pic(\frac{D}{p},N)}$ & \small{$\CM_{D,N}(R)\rightarrow\sqcup_{i=1}^2\CM_{\frac{D}{p},N}(R)$} \\
\cline{2-5}
\small{regular} & $p\parallel N$ & $p$ \small{splits in} $K$  & $\small{\bigsqcup_{i=1}^2\Pic(D,\frac{N}{p})}$ & \small{$\CM_{D,N}(R)\rightarrow\sqcup_{i=1}^2\CM_{D,\frac{N}{p}}(R)$}\\
\hline
& & & & \\
\hline
&  & Condition on $p$ &  $S$=Singular Points & $\Phi $  \\
\cline{2-5}
$\widetilde P$ & $p\mid D$ & $p$ \small{ramifies in} $K$ & $\Pic(\frac{D}{p},Np)$ & $\CM_{D,N}(R)\rightarrow\CM_{\frac{D}{p},Np}(R)$ \\
\cline{2-5}
\small{singular} & $p\parallel N$ & $p$ \small{ramifies in} $K$  & $\Pic(Dp,\frac{N}{p})$ & $\CM_{D,N}(R)\rightarrow\CM_{Dp,\frac{N}{p}}(R)$\\
\hline
\end{tabular}
\vspace{3mm}

Let us summarize here a weak, simplified version of some of the main results of this note, Theorems \ref{redCMsing}, \ref{CMsingthm}, \ref{CMsmooththm} and \ref{DRsing}.  We keep the notation as above.

\begin{theorem}
\begin{enumerate}
\item[(1)] Let $p\mid D$.
\begin{itemize}
\item [(i)] A Heegner point $P \in \CM(R)$ of $X_0(D,N)$ reduces to a singular point of $\tilde{X}_0(D,N)$ if and only if $p$ ramifies in $K$.

\item [(ii)] Assume $p\mid \disc(K)$ and $N$ is square-free. Then there is a one-to-one correspondence $\Phi: \CM(R) \stackrel{\simeq}{\lra} \CM_{D/p, Np}(R)$ which is equivariant for the action of $\Pic(R)$ and $W(D,N)$.

\item [(iii)] Assume $p\nmid \disc(K)$ and $N$ is square-free. Then there is a one-to-one correspondence $\Phi: \CM(R) \stackrel{\simeq}{\lra} \CM_{D/p, N}(R) \sqcup \CM_{D/p, N}(R)$ which is equivariant for the action of $\Pic(R)$ and $W(D,N)$.
\end{itemize}

\item[(2)] Let $p\mid (N,\disc(K))$ and $N$ be square-free. Then there is a one-to-one correspondence $\Phi: \CM(R) \stackrel{\simeq}{\lra} \CM_{Dp, N/p}(R)$ which is equivariant for the action of $\Pic(R)$ and $W(D,N)$.
\end{enumerate}
\end{theorem}

The above statements are indeed a weak version of our results, specially because we limited ourselves to claim the mere existence of the correspondences $\Phi $, whereas in \S \ref{CD} and \S \ref{DR} we actually provide explicit computable description of them in terms of bimodules. Combined with the results of P. Michel \cite[Theorem 10]{Mich} and Shimura's reciprocity law, we also obtain the following equidistribution result, which the reader may like to view as complementary to those of Cornut-Vatsal and Jetchev-Kane.

\begin{corollary}

\begin{itemize}
\item [(i)] Assume $p\parallel DN$ and $p\mid \disc(K)$. Let $\tilde X_0(D,N)_{\rm{sing}}=\{s_1,s_2\cdots,s_h\}$ be the set of singular points of $\tilde X_0(D,N)$ and let $\Pi:\CM(R)\ra\tilde X_0(D,N)_{\rm{sing}}$ denote the specialization map. Then, as $\disc(K)\ra\infty$, the sets $\CM(R)$ are equidistributed in $\tilde X_0(D,N)_{\rm{sing}}$ relatively to the measure given by
    \[\mu(s_i)=\omega_i^{-1}/(\sum_{j=1}^h\omega_j),\]
    where $\omega_i$ stands for the thickness of $s_i$. More precisely, there exists an absolute constant $\eta>0$ such that
    \[\frac{\#\{P\in\CM(R),\;\;\Pi(P)=s_i\}}{\#\CM(R)}=\mu(s_i)+O(\disc(K)^{-\eta}).\]
\item [(ii)] Assume $p\mid D$ and is inert in $K$. Let $\tilde X_0(D,N)_{cc}=\{c_1,c_2\cdots,c_t\}$ be the set of connected components of $\tilde X_0(D,N)$ and let $\Pi_c:\CM(R)\ra\tilde X_0(D,N)_{cc}$ denote the map which assigns to a point $P\in CM(R)$ the connected component where its specialization lies. Then, as $\disc(K)\ra\infty$, the sets $\CM(R)$ are equidistributed in $\tilde X_0(D,N)_{cc}$ relatively to the measure given by
    \[\mu(s_i)=\omega_i^{-1}/(\sum_{j=1}^h\omega_j),\]
    where $\omega_i$ stands for the cardinal of the group of unites of the order in $\Pic(D/p,N)$ attached to $c_i$. More precisely, there exists an absolute constant $\eta>0$ such that
    \[\frac{\#\{P\in\CM(R),\;\;\Pi_c(P)=c_i\}}{\#\CM(R)}=\mu(s_i)+O(\disc(K)^{-\eta}).\]
\end{itemize}
\end{corollary}

The original motivation of this research was the case of specialization of Heegner points $P\in \CM(R)$ on Cerednik-Drinfeld's special fiber $\tilde{X}_0(D,N)$ at a prime $p\mid D$, particularly when $p$ ramifies in $R$. There are several reasons which make this scenario specially interesting (beyond the fact that so far it had not been studied at all, as most articles on the subject exclude systematically this case):

\begin{enumerate}

\item[(1)] Let $f$ be a newform of weight $2$ and level $L\geq 1$ such that there exists a prime $p\mid \mid L$ which ramifies in $K$. Assume $L$ admits a factorization $L=DN$ into coprime integers $D$ and $N$ where: $p\mid D$ and $D/p$ is the square-free product of an odd number of primes, none of which splits in $K$; all prime factors of $N$ split in $K$. This is a situation in which little is known about the Birch and Swinnerton-Dyer conjecture for the abelian variety $A_f$ attached to $f$ by Shimura, when looked over $K$ and over its class fields.

It follows that for every order $R$ in $K$ of conductor $c\geq 1$ such that $(c,L)=1$, $\CM(R)\subset X_0(D,N)(K^{ab})$ is a {\em nonempty} subset of Heegner points.  Points $P\in \CM(R)$ specialize to {\em singular} points on the special fiber $\tilde{X}_0(D,N)$ at $p$; degree-zero linear combinations of them yield points on the N\'eron model $\mathcal J$ of the Jacobian of $X_0(D,N)$ over $\Z_p$, which can be projected to the group $\Phi_p$ of connected components of $\tilde{\mathcal J}$. It is expected that the results of this paper will be helpful in the study of the position of the image of such Heegner divisors in $\Phi_p$ in analogy with the work of Edixhoven in \cite{Edix} (though the setting here is quite different from his).

Moreover, if one further assumes that the sign of the functional equation of $f$ is $+1$ (and this only depends on its behavior at the prime factors of $D$), it is expected that this can be related to the special values $L(f/K,\chi,1)$ of the L-function of $f$ over $K$ twisted by finite characters unramified at $p$, as an avatar of the Gross-Zagier formula in the spirit of \cite{BerDar}. If true, a system of Kolyvagin classes could be constructed out of the above mentioned Heegner divisors and a word could be said on the arithmetic of $A_f$. The author hopes to pursue these results in the work in progress \cite{MoRo}.

\item[(2)] On the computational side, there is an old, basic question which seems to remain quite unapproachable: can one write down explicit equations of curves $X_0(D,N)$ over $\Q $ {\em when $D>1$}? As it is well-known, classical elliptic modular curves $X_0(N):=X_0(1,N)$ can be tackled thanks to the presence of cusps, a feature which is only available in the case $D=1$. Ihara \cite{Ihara} was probably one of the first to express an interest on this problem, and already found an equation for the genus $0$ curve $X_0(6,1)$, while challenged to find others. Since then, several authors have contributed to this question (Kurihara, Jordan, Elkies, Clark-Voight for genus $0$ or/and $1$, Gonzalez-Rotger for genus $1$ and $2$). The methods of the latter are heavily based on Cerednik-Drinfeld's theory for $\mathcal{X}_0(D,N)\times \Z_p$ for $p\mid D$ and the arithmetic properties of fixed points by Atkin-Lehner involutions on $X_0(D,N)$. It turns out that these fixed points are usually Heegner points associated to fields $K$ in which at least one (sometimes all!) prime $p\mid D$ ramifies: one (among others) of the reasons why the methods of \cite{GoRo1}, \cite{GoRo2} do not easily extend to curves of higher genus is the little understanding one has of the specialization of these points on the special fiber $\tilde{X}_0(D,N)$ at $p$. It is hoped that this article can partially cover this gap and be used to find explicit models of many other Shimura curves: details for hyperelliptic curves may appear in \cite{Mol}. For this application, our description of the geometrically-constructed maps $\Phi$ in pure algebraic terms by means of Ribet's bimodules turns out to be crucial, as this allows to translate it into an algorithm.

\end{enumerate}

\subsubsection*{Notation.}
Throughout, for any module $M$ over $\Z$  and for any prime $p$ we shall write $M_p=M\otimes_{\Z} \Z_p$. Similarly, for any homomorphism $\chi : M \ra N$ of modules over $\Z $, we shall write $\chi_p: M_p \ra N_p$ for the natural homomorphism obtained by extension of scalars. We shall also write $\hat\Z$ to denote the profinite completion of $\Z$, and $\hat M=M \otimes_\Z \hat\Z$.

For any $\Z $-algebra $\mathcal D$, write $\mathcal D^0=\mathcal D\otimes _{\Z} \Q $ and say that an embedding $\varphi: \mathcal D_1\, \hookrightarrow \,\mathcal D_2$ of $\Z$-algebras is {\em optimal} if $\varphi(\mathcal D^0_1)\cap\mathcal D_2=\varphi(\mathcal D_1)$ in $\mathcal D^0_2$.

\subsubsection*{Acknowledgements.}
The author would like to thank Massimo Bertolini, Matteo Longo and specially Josep González and Víctor Rotger for their comments and discussions
 throughout the development of this paper. The author also thanks Xevi Guitart and Francesc Fité for carefully reading a draft of the paper and providing many helpful
comments.

\section{Preliminaries}\label{Adelesembdd}

\subsection{Quaternion orders and  optimal embeddings}\label{QOOE}

Let $B$ be a quaternion algebra over $\Q$, of reduced discriminant $D$.
An order $\cO $ in $B$ is Eichler if it is the intersection of two maximal orders. Its index $N$ in any of the two maximal orders is called its level.

An orientation on $\cO $  is a collection of choices, one for each prime $p\mid D N$: namely, a choice of a homomorphism $o_p: \cO\rightarrow \F_{p^2}$ for each prime $p\mid D$, and a choice of a local maximal order
$\tilde{\cO}_p$ containing  $\cO_p$ for each $p\mid N$. One says that two oriented Eichler orders $\cO, \cO'$ are isomorphic whenever there exists an automorphism $\chi $ of $B$ with $\chi ( \cO ) = \cO'$ such that $o_p = o'_p \circ \chi $ for $p\mid D$ and $\chi_p(\tilde{\cO}_p) = \tilde{\cO'}_p$ for $p\mid N$. Fix an oriented Eichler order $\cO$ in $\cE(D,N)$.

\begin{definition}\label{*}
For a projective left $\cO$-ideal $I$ in $B$, let $I\ast \cO $ be the right order $\cO'$ of $I$ equipped with the following local orientations at $p\mid D N$. Since $I$ is locally principal, we may write $I_p = \cO_p \alpha_p$ for some $\alpha_p\in B_p$, so that $\cO'_p = \alpha_p^{-1}\cO_p \alpha_p$. For $p\mid D$, define $o'_p(\alpha_p^{-1} x\alpha_p)= o_p(x)$. For $p\mid N$, define $\tilde{\cO}'_p= \alpha_p^{-1}\tilde{\cO }\alpha_p$.
\end{definition}

As in the introduction, let $R$ be an order in an imaginary quadratic field $K$. Two local embeddings $\phi_p,\varphi_p:R_p\hookrightarrow \cO_p$ are said to be equivalent, denoted $\varphi_p\sim_p\phi_p$, if there exists  $\lambda\in\cO_p^{\times}$ such that
$\phi_p=\lambda\varphi_p\lambda^{-1}$. Let us denote by $m_p$ the number of equivalence classes of such local embeddings.
Note that, in case $p\mid D$ or $p\parallel N$, we have $m_p\in\{0,1,2\}$ by \cite[Theorem 3.1, Theorem 3.2]{Vig}.
We say that two global embeddings  $\phi,\varphi:R\hookrightarrow \cO$ are locally
equivalent  if $\varphi_p\sim_p\phi_p$ for all primes  $ p$.

By \cite[Theorem 5.11]{Vig},
there is a faithful action of $\Pic(R)$ on $\CM_{D,N}(R)$, denoted $[J]\ast\varphi$ for $[J]\in\Pic(R)$ and $\varphi\in \CM_{D,N}(R)$. It can be explicitly defined as follows: if $\varphi: R\hookrightarrow \cO$, let $[J]\ast\cO=\cO\varphi(J)\ast\cO$ and let $[J]\ast\varphi : R \hookrightarrow [J]\ast\cO$ be the natural optimal embedding
\[
R\stackrel{\varphi}{\hookrightarrow}\{x\in B,\;:\;\cO\varphi(J)x\subseteq\cO\varphi(J)\}.
\]

The action of $\Pic(R)$ preserves local equivalence, and in fact
\begin{equation}\label{h}
\# \CM_{D,N}(R)=h(R)\prod_{p\mid DN} m_p,
\end{equation} where $h(R)=\#\Pic(R)$ is the class number of $R$.

For each  $p^n\parallel DN$, there is also an Atkin-Lehner involution $w_{p^n}$ acting on $\CM_{D, N}(R)$, which can be described as follows.
Let $\dP_\cO$ denote the single two-sided ideal of $\cO$ of norm $p^n$. Notice that $\dP_\cO\ast\cO$ equals $\cO $ as orders in $B$, but they are endowed with possibly different local orientations. Given
$\varphi\in \CM_{\cO }(R) \subset \CM_{D,N}(R)$, $w_{p^n}$ maps $\varphi$ to the
optimal embedding $w_{p^n}(\varphi):R\hookrightarrow \dP_\cO\ast\cO$, where $w_{p^n}(\varphi)$ is simply $\varphi$ as ring homomorphism.

In the particular case $p\parallel DN$ and $m_p=2$, the involution  $w_{p}$ switches the two
local equivalence classes at $p$.

For any $m\parallel DN$ we denote by $w_m$ the composition $w_m=\prod_{p^n\parallel m}w_{p^n}$. The set
$W(D,N)=\{w_m\;:\;m\parallel DN\}$ is  an abelian $2$-group with the operation $w_m w_n=w_{nm/(m,n)^2}$, called the Atkin-Lehner group.

Attached to the order $R$, we set
\[
D(R)=\prod_{p\mid D,m_p=2}p\,,\quad N(R)=\prod_{p^n\parallel N,m_p=2}p,
\]
and
\[
W_{D,N}(R)=\{w_m\in
W(D,N)\;:\;m\parallel D(R)N(R)\}.
\]

Suppose that $N$ is square free. Since $\Pic(R)$ acts faithfully on $\CM_{D,N}(R)$ by preserving local equivalences, and Atkin-Lehner
involutions $w_{p}\in  W(R)$ switch local equivalence classes at $p$, it follows from \eqref{h} that the group
$W_{D,N}(R)\times\Pic(R)$ acts freely and transitively on the set $\CM_{D,N}(R)$.


\subsection{Shimura curves}\label{ModCM}

Assume, only in this section, that $B$ is indefinite. By an abelian surface with quaternionic multiplication (QM) by $\cO$
over a field $K$, we mean a pair $(A, i)$ where
\begin{itemize}
\item [i)] $A/K$ is an abelian surface.
\item [ii)]  $i:\cO\hookrightarrow\End(A)$ is an optimal embedding.
\end{itemize}

\begin{remark}
The optimality condition $i(B)\cap\End(A)= i(\cO)$ is always satisfied when $\cO$ is maximal.
\end{remark}

For such a pair, we denote by $\End (A,i)$ the ring of endomorphisms which commute with
$i$, i.e., $\End(A,i)=\{\phi\in\End(A):\phi \circ i(\alpha)=i(\alpha)\circ\phi\mbox{  for all }\alpha\in\cO\}$.

Let us  denote by $X_0(D,N)/\Z $ Morita's integral model (cf.\,\cite{BC}, \cite{Mo}) of the Shimura curve associated to $\cO$. As Riemann surfaces,  $X_0(D,N)_{\C} = \Gamma_0(D, N)\backslash \mathcal H$ if $D>1$; $X_0(D,N)_{\C} = \Gamma_0(D, N)\backslash (\mathcal H \cup \PP^1(\Q))$ if $D=1$. Fixing an isomorphism $B\otimes \R \simeq \M_2(\R )$, the moduli interpretation of $X_0(D,N)$ yields a one-to-one correspondence
\[
\begin{matrix}
\Gamma_0(D, N)\backslash \mathcal H & \stackrel{1:1}{\longleftrightarrow }& \left\{ \begin{array}{l}
\mbox{ Abelian surfaces } (A, i)/\C \mbox{ with} \\
\mbox{quaternionic multiplication by } \cO \\
\end{array} \right\}/\cong\,\\
[\tau ] & \mapsto & [A_{\tau }= \C^2/\Lambda_{\tau}, \, i_{\tau}]
\end{matrix}
\]
where $i=i_{\tau }$ arises from the natural embedding $\cO \hookrightarrow B \hookrightarrow \M_2(\R )$ and $\Lambda_{\tau } = i(\cO) \cdot v$ for $v=\begin{pmatrix}  \tau \\ 1\end{pmatrix}$.

Above, two pairs $(A,i)$ and  $(A',i')$ are isomorphic if there is an isomorphism $\phi:A\longrightarrow A'$
such that $\phi\circ i(\alpha)=i'(\alpha)\circ\phi$ for all $\alpha\in\cO$. Throughout, we shall denote by $P=[A,i]$ the isomorphism class of $(A,i)$, often regarded as a point on $X_0(D, N)$.

A (non-cusp) point $P=[A,i]\in X_0(D,N)(\C)$ is a Heegner (or CM) point by $R$ if $\End(A,i)\simeq R$. We shall denote by $\CM(R)$ the set of such points. By the theory of complex multiplication, we actually have $\CM(R)\subset X_0(D,N)(\bar\Q)$. As it is well known, there is a one-to-one correspondence

\begin{equation}\label{phip}
\begin{matrix}
\CM(R) & \stackrel{\sim }{\to} &  \CM_{D,N}(R) \\
P=[(A, i)] & \mapsto & \varphi_P
\end{matrix}
\end{equation}
where $$\varphi_P: R\simeq\End(A,i)\hookrightarrow\End_{\cO}(H_1(A,\Z))\cong\cO.$$
Throughout, we shall fix the isomorphism $R\simeq\End(A,i)$, to be the canonical one described in \cite[Definition 1.3.1]{Jor}.

Let $P=[A, i]=[A_{\tau}, i_{\tau}]\in\CM(R)$. Via $\varphi_P$, we may regard $\cO$ as a locally free right $R$-module of rank 2. As such,
\[\cO \simeq R\oplus e I,\]
for some $e\in B$ and some locally free $R$-ideal $I$ in $K$. This decomposition allows us to decompose
\[
\Lambda_{\tau}= i(\cO) v = i(\varphi_P R) v\oplus e \cdot i(\varphi_P I) v.
\]

Since $R=\End(A, i)$ and $i(\cO )\otimes \R = \M_2(\R )$, it follows that $i(\varphi_P K) \subset \C \stackrel{\mathrm{diag}}{\hookrightarrow } \M_2(\C)$. Hence $\Lambda_{\tau} = i(\varphi_P R) v\oplus i(\varphi_P I) e v$ and $A$ is isomorphic to the product of two  elliptic curves with CM by $R$, namely $E=\C/R$ and $E_I=\C/I$.
Moreover, the action of $\cO$ on $E\times E_I$ induces the natural left action of $\cO$ on $R\oplus e I$.

\section{Bimodules: Ribet's work}\label{RibWork}

\subsection{Admissible bimodules}\label{Admbim}

Let $Z$ denote either $\Z$ or $\Z_p$ for some prime $p$. Let $\cO$ be an Eichler order and let $\cS$ be a maximal order, both over $Z$, in two possibly distinct quaternion algebras $B$ and $H$. By a $(\cO,\cS)$-bimodule $\cM$ we mean a free module of finite rank over $Z$ endowed with structures of left projective $\cO$-module and right (projective) $\cS$-module. For any $(\cO,\cS)$-bimodules $\cM$ and $\dN$, let us denote by $\Hom_{\cO}^{\cS}(\cM,\dN)$ the set of $(\cO,\cS)$-bimodule homomorphisms from $\cM$ to $\dN$, i.e., $Z$-homomorphisms equivariant for the left action of $\cO$ and the right action of $\cS$. If $\cM=\dN$, we shall write $\End_{\cO}^{\cS}(\cM)$ for $\Hom_{\cO}^{\cS}(\cM,\cM)$.

Notice that, since $\cS $ is maximal, it is hereditary and thus all $\cS$-modules are projective. Note also that $(\cO,\cS)$-bimodules are naturally $\cO \otimes \cS$-modules.

Let $\Sigma$ be the (possibly empty) set of prime numbers which ramify in both $\cO$ and $\cS$. For $p$ in $\Sigma $, let $\dP_\cO$ and $\dP_\cS$ denote the unique two-sided ideals of $\cO$ and $\cS$, respectively, of norm $p$.

\begin{definition}
An $(\cO,\cS)$-bimodule $\cM$ is said to be \emph{admissible} if $\dP_{\cO}\cM=\cM\dP_{\cS} \mbox{  for all  }p\in\Sigma$.
\end{definition}

\begin{remark}\label{bimodM2}
Let $\cM$ be a $(\cO,\cS)$-bimodule of rank $4n$ over $Z$, for some $n\geq 1$. Since $\cM$ is free over $Z$, $\cM$ is also free over $\cS$ as right module, by \cite{Eich}. Up to choosing an isomorphism between $\cM$ and $\cS^n$, there is a natural identification $\End_{\cS}(\cM)\simeq M(n,S)$. Thus, giving a structure of left $\cO$-module on $\cM$ amounts to giving a homomorphism $f:\cO\rightarrow M(n,\cS)$.

Since Eichler orders are Gorenstein (cf.\,\cite{Jsson}), the $\cO$-module $\cM$ is projective if and only if $f$ is optimal. In particular, we conclude that the isomorphism class of a $(\cO, \cS )$-bimodule $\cM$ is completely determined by the $GL(n,S)$-conjugacy class of an optimal embedding $f:\cO\rightarrow M(n,\cS)$. Finally, in terms of $f$, $\cM$ is admissible if and only if $f(\dP_{\cO})= M(n,\dP_{\cS})$ for all $p \in\Sigma$.
\end{remark}


We now proceed to describe Ribet's classification of -local and global- admissible bimodules.
Let $p$ be a prime. Let $\cO$ denote the maximal order in a quaternion division algebra $B$ over $\Q_p$. Let $\dP = \cO\cdot \pi $ be the maximal ideal of $\cO$ and let $\F_{p^2}$ be the residue field of $\dP$.

\begin{theorem}{\cite[Theorem 1.2]{Rib}}
Assume $Z=\Z_p$, $B$ is division over $\Q_p$ and $\cO$ is maximal in $B$.
Let $\cM$ be an admissible $(\cO,\cO)$-bimodule of finite rank over $\Z_p$.
Then
\[\cM\cong\underbrace{\cO\times\dots\times\cO }_{r\mbox{ factors} }\times\underbrace{\dP\times\dots\times\dP }_{s\mbox{ factors} },\]
regarded as a bimodule via the natural action of $\cO$ on $\cO$ itself and on $\dP$ given by left and right multiplication.
In that case, we say that a $(\cO,\cO)$-bimodule $\cM$ is of type $(r,s)$
\end{theorem}

Using the above local description, Ribet classifes global bimodules of rank 8 over $\Z$ in terms of their algebra of endomorphisms, {\em provided $\cO $ is maximal.} Hence, assume for the rest of this subsection that $Z=\Z$ and $\cO$ is maximal in $B$. Write $D_B=\disc(B)$, $D_H=\disc(H)$ for their reduced discriminants. Let $\cM$ be a $(\cO,\cS)$-bimodule of rank 8, set
\[D_0^\cM=\prod_{p\mid D_B D_H,p\not\in\Sigma} p\]
and let $\mathcal{C}$ be the quaternion algebra over $\Q $ of reduced discriminant $D_0^\cM$. Note that the class of $\mathcal C$ in the Brauer group $\Br(\Q)$ of $\Q $ is the sum of the classes of the quaternion algebras $B$ and $H$.

We shall further assume for convenience that $D_0^\cM\ne 1$, i.\,e., $B\not \simeq H$.

\begin{proposition}\label{EndbiMod}\cite[Proposition 2.1]{Rib}
$\End_{\cO}^{\cS}(\cM)\otimes\Q\,\simeq \,\mathcal{C}$ and the ring $\End_{\cO}^{\cS}(\cM)$ is an Eichler order in $\mathcal{C}$ of level $\prod_k p_k$, where $p_k$ are the primes in $\Sigma$ such that $\cM_{p_k}$ is of type $(1,1)$.
\end{proposition}

Assume now that $\cO$ and $\cS$ are equipped with orientations. According to \cite[p.\,17]{Rib}, the Eichler order $\Lambda=\End_{\cO}^{\cS}(\cM)$ is endowed with orientations in a natural way, and the isomorphism class of $\cM$ is determined by the isomorphism class of $\Lambda$ as an oriented Eichler order.

Let $N_0^\cM$ be the product of the primes $p_k$ in $\Sigma$ such that $\cM_{p_k}$ is of type $(1,1)$.

\begin{theorem}\label{1-1bimod}{\cite[Theorem 2.4]{Rib}}
The map $\cM\mapsto \Lambda$ induces a one-to-one correspondence between the set of isomorphism classes of admissible rank-8 bimodules  of type $(r_p,s_p)$ at $p\in\Sigma$, and the set $\Pic(D_0^\cM,N_0^\cM)$ of isomorphism classes of oriented Eichler orders.
\end{theorem}

\subsection{Supersingular surfaces and bimodules}\label{ss-bim}

Let $p$ be a prime and let $\F$ be a fixed algebraic closure of $\F_p$.
An abelian surface $\widetilde A/\F$ is \emph{supersingular} if it is isogenous to a product of supersingular elliptic curves over $\F$. Given a supersingular abelian surface $\widetilde A$, one defines \emph{Oort's invariant} $a(\widetilde A)$ as follows. If $\widetilde A$ is isomorphic to a product of supersingular elliptic curves, set $a(\widetilde A)=2$; otherwise set $a(\widetilde A)=1$ (see \cite[Chapter 1]{Oort} for an alternative definition of this invariant).


Let $B$ be an indefinite quaternion algebra over $\Q $, of reduced discriminant $D$. Let $N\geq 1$, $(N,pD)=1$ and let $\cO \in \cE(D,N)$.

In this subsection we shall consider pairs $(\widetilde A,\widetilde i)$ of supersingular abelian surfaces with QM by $\cO$ over $\F$ such that $a(\widetilde A)=2$. By a theorem of Deligne (cf.\,\cite[Theorem 3.5]{Shi}) and Ogus \cite[Theorem 6.2]{Ogu}, $\widetilde{A}\cong \widetilde E\times \widetilde E$, where $\widetilde E$ is any fixed supersingular elliptic curve over $\F$.

Let us denote by $\cS=\End(\widetilde E)$ the endomorphism ring of $\widetilde E$. According to a well known theorem of Deuring \cite{Deu}, $\cS$ is a maximal order in the quaternion algebra $H=\cS\otimes \Q$, which is definite of discriminant $p$.
Therefore, giving such an abelian surface $(\widetilde{A},\widetilde{i})$ with QM by $\cO$ is equivalent to providing an optimal embedding
\[\widetilde{i}:\cO\longrightarrow M(2,\cS)\simeq\End(\widetilde{A}),\]
or, thanks to Remark \ref{bimodM2}, a $(\cO,\cS)$-bimodule $\cM$ of rank 8 over $\Z$.

\begin{remark}\label{OrientationS} The maximal order $\cS$ comes equipped with a natural orientation at $p$, and therefore can be regarded as an element of $\E(p,1)$. This orientation arises from the action of $\cS$ on a suitable  quotient of the Dieudonn\'e module of the elliptic curve $\widetilde E$, which turns out to be described by a character $\cS\rightarrow \F_{p^2}$. See \cite[p.\,37]{Rib} for more details.
\end{remark}

Let $\cM=\cM_{(\widetilde{A},\widetilde{i})}$ be the bimodule attached to the pair $(\widetilde{A},\widetilde{i})$ by the above construction. Then
\begin{equation}\label{end}
\End_{\cO}^{\cS}(\cM) \simeq \{\gamma\in M(2,\cS)\simeq\End(\widetilde{A})\;\mid\;\gamma\circ\widetilde{i}(\alpha)=\widetilde{i}(\alpha)\circ\gamma,\;\mbox{for all } \alpha\in\cO\}=\End(\widetilde{A},\widetilde{i}).
\end{equation}

The above discussion allows us to generalize Ribet's Theorem \ref{1-1bimod}  to $(\cO,\cS)$-bimodules where $\cO$ is  an Eichler order in $B$ of level $N$, $(N,p)=1$, not necessarily maximal. Keep the notation $\Sigma$, $D_0^\cM$ and $N_0^\cM$ as in \S\ref{Admbim}.

\begin{theorem}\label{1-1bimodref}
The map $\cM\mapsto \End_\cO^\cS(\cM)$ induces a one-to-one correspondence between the set of isomorphism classes of admissible $(\cO,\cS)$-bimodules $\cM$ of rank $8$ over $\Z $ and of type $(r_p,s_p)$ at $p\in\Sigma$, and the set $\Pic(D_0^\cM,NN_0^\cM)$ of isomorphism classes of oriented Eichler orders.
\end{theorem}
\begin{proof} Given $\cM$ as in the statement, let $(\widetilde{A},\widetilde{i})$ be the abelian surface with QM by $\cO $ over $\F$ attached to $\cM$ by the preceding discussion. As explained in Appendix A, there is a one-to-one correspondence between isomorphism classes of such pairs $(\widetilde{A},\widetilde{i})$ and triples $(\widetilde A_0,\widetilde i_0, C)$, where $(\widetilde A_0,\widetilde i_0)$ has QM by a maximal order $\cO_0\supseteq\cO$ and $C$ is a $\Gamma_0(N)$-level structure.

Using Theorem \ref{1-1bimod}, Ribet proves that $\End(\widetilde A_0,\widetilde i_0, C)\subseteq\End(\widetilde A_0,\widetilde i_0)$ is an Eichler order of level $NN_0$ and that such triples are characterized by their class in $\Pic(D_0^\cM,NN_0^\cM)$. We refer the reader to \cite[Theorem 4.15]{Rib} for $p\mid D$ and to \cite[Theorem 3.4]{Rib} for $p\nmid D$.

Finally, $\End(\widetilde{A},\widetilde{i})=\End(\widetilde A_0,\widetilde i_0, C)$ by Proposition \ref{ChangInt} in Appendix A. This yields the desired result.
\end{proof}

\section{Supersingular specialization of Heegner points}\label{ssesp}

As in the previous section, let $p$ be a prime and let $\F$ be a fixed algebraic closure of $\F_p$. Let $B$ be an indefinite quaternion algebra over $\Q$ of discriminant $D$. Let $N\geq 1,\;(D,N)=1$ and $\cO\in \cE(D,N)$. Let $P=[A,i]\in X_0(D,N)(\overline\Q)$ be a (non-cusp) point on the Shimura curve $X_0(D,N)$.
Pick a field of definition  $M$ of $(A,i)$.

Fix a prime $\mathfrak{P}$ of $M$ above $p$ and let $\widetilde{A}$ denote the special fiber of the N\'eron model of $A$ at $\dP$. By \cite[Theorem 3]{Rib2}, $A$ has potential good reduction at $\dP$. Hence, after extending the field $M$ if necessary, we obtain that $\widetilde{A}$ is smooth over $\F$.

Since $A$ has good reduction at $\dP$, the natural morphism
$\phi:\End(A)\hookrightarrow\End(\widetilde{A})$
is injective. The composition $\widetilde{i}=\phi\circ i$ yields an optimal embedding
$\widetilde{i}:\cO\hookrightarrow\End(\widetilde{A})$ and the isomorphism class of the pair $(\widetilde{A},\widetilde{i})$ corresponds to the reduction $\tilde P$ of the point $P$ on the special fiber $\widetilde{X}_0(D,N)$ of $\cX_0(D,N)/\Z$ at $p$.

Let us denote by $\phi_P:\End(A,i)\hookrightarrow\End(\widetilde{A},\widetilde{i})$ the restriction of $\phi$ to $\End(A,i)$.

\begin{lemma}\label{lemaopt}
The embeddings $\phi$ and $\phi_P$ are optimal locally at every prime but possibly at $p$.
\end{lemma}

\begin{proof} We first show that $\phi_\ell$ is optimal for any prime $\ell\neq p$. For this, let us identify $\End^0(A)_\ell$ with a subalgebra of $\End^0(\tilde A)_\ell$ via $\phi $, so that we must prove that $\End^0(A)_\ell\cap\End(\widetilde{A})_\ell = \End(A)_\ell$.
It is clear that $\End^0(A)_\ell\cap\End(\widetilde{A})_\ell\supset \End(A)_\ell$.
As for the reversed inclusion, let $\alpha\in\End^0(A)_\ell\cap\End(\widetilde{A})_\ell$ and let $M$ be a field of definition of $\alpha$.
Due to good reduction, there is an isomorphism of Tate modules $T_\ell(A)=T_\ell(\widetilde{A})$. Faltings's theorem on Tate's Conjecture \cite[Theorem 7.7]{Sil} asserts that $\End_M(A)_\ell=\End_{G_{M}}(T_\ell(A))\subseteq\End(T_\ell(A))$, where $G_M$ stands for the absolute Galois group of $M$ and $\End_{G_{M}}(T_\ell(A))$ is the subgroup of $\End(T_\ell(A))$ fixed by the action of $G_M$. Since $\alpha\in\End^0_M(A)_\ell$, there exists $n\in\Z$ such that $n\alpha\in\End_M(A)_\ell=\End_{G_{M}}(T_\ell(A))$. By hypothesis $\alpha\in\End_{\F}(\widetilde{A})_\ell\subseteq \End(T_\ell(A))$, hence it easily follows that $\alpha\in\End_{G_{M}}(T_\ell(A))=\End_M(A)_\ell\subseteq\End(A)_\ell.$

This shows that $\phi_\ell$ is optimal. One easily concludes the same for $(\phi_P)_\ell$ by taking into account that
$\End(A,i)=\End(\widetilde{A},\widetilde{i})\cap \End(A)$.
\end{proof}

\begin{remark}\label{remlemaopt}
Notice that if $\End(A,i)$ is maximal in $\End^0(A,i)$, the embedding $\phi_P$ is optimal.
In fact, if $\End(A,i)$ is not maximal in $\End^0(A,i)$ the embedding $\phi_P$ may not be optimal at $p$. For example, let $R$ be an order in an imaginary quadratic field $K$ of conductor $cp^r$, let $A\simeq E\times E$ where $E$ is an elliptic curve with CM by $R$ and let $i:M_2(\Z)\hookrightarrow M_2(R)=\End(A)$. Then $\End(A,i)\simeq R$ whereas, if $p$ is inert in $K$, $\End(\tilde A,\tilde i)\in\E(p,1)$. Thus if $\phi_P$ was optimal, it shall provide an element of $\CM_{p,1}(R)$ which is impossible by \cite[Corollaire 5.12]{Vig}.
\end{remark}


Let $P=[A,i]\in \CM(R)\subset X_0(D,N)(\C )$ and denote by $c$ the conductor of $R$. In \S\ref{ModCM} we proved that $A\cong E\times E_I$ where $E=\C/R, E_I=\C/I$ are elliptic curves with $\CM$ by $R$. Here $I$ is a projective $R$-ideal in $K$.

Assume that $p$ is coprime to $cN$ and does not split in $K$. Then $\widetilde{A}\simeq \widetilde{E}\times\widetilde{E_I}$ is a product of supersingular elliptic curves over $\F$. Let $\cS\in\E(p,1)$ be the endomorphism ring of $\widetilde{E}$ endowed with the natural orientation described in \S\ref{ss-bim}. Since $a(\widetilde A)=2$, we can assign a $(\cO,\cS)$-bimodule $\cM=\cM_{\widetilde P}$ to $\widetilde P=[\widetilde A,\widetilde i]$ as in the previous section.

\begin{theorem}\label{CMbimod}

\begin{enumerate}

\item[(a)] There exists an optimal embedding $\psi:R\hookrightarrow\cS$ such that, for all $P\in\CM(R)$, \begin{equation}\label{MOS} \cM_{\widetilde P}\simeq \cO\otimes_R\cS,\end{equation} where $\cS$ is regarded as left $R$-module via $\psi$ and $\cO$ as right $R$-module via $\varphi_P$.

\item[(b)] Upon the identifications \eqref{end} and \eqref{MOS}, the optimal embedding $\phi_P$ is given by the rule
\[
\begin{array}{rcc}
R & \hookrightarrow & \End_{\cO}^{\cS}(\cO\otimes_R\cS) \\
\delta & \longmapsto & \phi_P(\delta):\alpha\otimes s  \mapsto  \alpha\delta\otimes s,
 \end{array}
\]
up to conjugation by $\End_{\cO}^{\cS}(\cO\otimes_R\cS)^{\times}$.
\end{enumerate}
\end{theorem}
\begin{proof}
As explained in \S \ref{ss-bim}, the isomorphism class of the bimodule $\cM_{\widetilde P}$ is completely determined by the optimal embedding $\widetilde{i}:\cO\longrightarrow \End(\widetilde{A})=M(2,\cS)$, which in turn is defined as the composition of $i$ with $\phi:\End(A)\hookrightarrow\End(\widetilde{A})$.

More explicitly, as described in \S\ref{ModCM}, the action of $\cO $ on $A\simeq E\times E_I$ given by the embedding $i:\cO\hookrightarrow \End(A)\simeq\End(E\times E_I)$ can be canonically identified with the action of $\cO$ on the right $R$-module $\cO \simeq R\oplus eI$ given by left multiplication $(\alpha, x)\mapsto \alpha\cdot x$, where here $\cO$ is regarded as a right $R$-module via $\varphi_P$. Therefore, the embedding $\widetilde{i}:\cO\longrightarrow M(2,\cS)$ is the one induced by the action of $\cO$ given by left multiplication on the $\cS$-module $(R\oplus eI)\otimes_R\cS=\cO\otimes_R\cS$, where $\cS$ is viewed as a left $R$-module via $\psi: \End(E)\simeq R\hookrightarrow\End(\widetilde{E})\simeq\cS $. This shows that $\cM_{\widetilde P}\cong\cO\otimes_R\cS$.

Finally, since the action of $\delta\in R\simeq \End(E\times E_I,i)$ on $E\times E_I$ is naturally identified with the action of $\delta$ on $\cO\simeq R\oplus eI$ given by right multiplication, the optimal embedding $R  \hookrightarrow  \End_{\cO}^{\cS}(\cO\otimes_R\cS)$ arising from $\phi_P:\End(A,i)\hookrightarrow \End(\widetilde A,\widetilde i)$ is given by $\phi_P(\delta)(\alpha\otimes s)= \alpha\delta\otimes s$.
This endomorphism clearly commutes with both (left and right) actions of $\cO$ and $\cS$. Moreover this construction is determined up to isomorphism of $(\cO,\cS)$-bimodules, i.e. up to conjugation by $\End_{\cO}^{\cS}(\cM_{\widetilde P})^{\times}$.
\end{proof}

\subsection{The action of $\Pic(R)$}

\begin{definition}
For a $(\cO,\cS)$-bimodule  $\cM$, let
\begin{itemize}
\item $\Pic_\cO^\cS(\cM)$ denote the set of isomorphism classes of $(\cO,\cS)$-bimodules that are locally isomorphic to $\cM$,

\item $\CM_{\cM}(R)$ denote the set of $\End_{\cO}^{\cS}(\cM)^{\times}$-conjugacy classes of optimal embeddings $\varphi: R\hookrightarrow \End_{\cO}^{\cS}(\cM)$,

\item $\CM_{\cO,\cS}^{\cM}(R) = \{( \dN,\psi):\; \dN\in\Pic_\cO^\cS(\cM),\psi\in \CM_{ \dN}(R)\}$.
\end{itemize}
\end{definition}

Let $P=[A,i]\in \CM(R)$ be a Heegner point and let $\cM_{\widetilde P}$ be the $(\cO,\cS)$-bimodule attached to its specialization $\widetilde P=[\widetilde A,\widetilde i]$ at $p$ as described above.

By Theorem \ref{CMbimod}, the bimodulde $\cM_{\widetilde P}$ and the $\End_{\cO}^{\cS}(\cM_{\widetilde P})^{\times}$-conjugacy class of $\phi_P:R\hookrightarrow\End_{\cO}^{\cS}(\cM_{\widetilde P})$ are determined by the optimal embedding $\varphi_P\in\CM_{D,N}(R)$.
Hence specialization at $p$ induces a map
\begin{equation}\label{Phi}
\phi:\CM_{D,N}(R)\ra\bigsqcup_\cM\CM_{\cO,\cS}^{\cM}(R),\quad\varphi_P\ra(\cM_{\widetilde P},\phi_{P}),
\end{equation}
where $\cM$ runs over a set of representatives of local isomorphism classes of $(\cO,\cS)$-bimodules arising from some $\varphi_P\in\CM_{D,N}(R)$. Note that composing $\phi$ with the natural projection
\begin{equation}\label{proj}
\pi:\bigsqcup_\cM\CM_{\cO,\cS}^{\cM}(R)\ra\bigsqcup_\cM\Pic_\cO^\cS(\cM),\quad (\cM,\psi)\ra \cM,
\end{equation}
one obtains the map $\varphi_P \mapsto \cM_{\widetilde P}$ which assigns to $\varphi_P$ the bimodule that describes the supersingular specialization of the Heegner point associated to it.


For any locally free rank-1 left $\cO$-module $I$, let us consider the right $\cO$-module $$I^{-1}=\{x\in B\;:\;Ix\subseteq \cO\}.$$
It follows directly from Definition \ref{*} that $I\ast\cO$ coincides with the left order of $I^{-1}$, endowed with the natural local orientations.

\begin{lemma}\label{auxactPic}
There is an isomorphism of $(I\ast\cO,I\ast\cO)$-bimodules between
$I\ast\cO$ and $ I^{-1}\otimes_{\cO}I$.
\end{lemma}
\begin{proof}
Since $I\ast\cO=\{x\in B\;:\;Ix\subseteq I\}=\{x\in B\;:\;xI^{-1}\subseteq I^{-1}\}$, both $I\ast\cO$ and $ I^{-1}\otimes_{\cO}I$ are $(I\ast\cO,I\ast\cO)$-bimodules with the obvious left and right $(I\ast\cO)$-action.
Moreover, by definition,
\[
I^{-1}\otimes_{\cO}I=\{x\in B\otimes_{\cO}I=B\;:\;Ix\subseteq \cO\otimes_{\cO}I=I\}=I\ast\cO,
\]
which proves the desired result.
\end{proof}

\begin{theorem}\cite[Theorem 2.3]{Rib}\label{1-1corrbim}
Let $\cM$ be a $(\cO,\cS)$-bimodule and let $\Lambda=\End_{\cO}^{\cS}(\cM)$. Then the
correspondence $ \dN\mapsto \Hom_{\cO}^{\cS}(\cM, \dN)$ induces a bijection between
$\Pic_\cO^\cS(\cM)$ and the set of isomorphism classes of locally free rank-1 right $\Lambda$-modules.
\end{theorem}

The one-to-one correspondence is given by:
\[
\begin{array}{ccc}
 \dN & \longrightarrow & I( \dN)=\Hom_{\cO}^{\cS}(\cM, \dN))\\
 \dN(I)=I\otimes_{\Lambda}\cM & \longleftarrow & I
\end{array}
\]

We can define an action of $\Pic(R)$ on $\CM_{\cO,\cS}^{\cM}(R)$ which generalizes the one on $\CM_{D,N}(R)$. Let $(\dN,\psi:R\hookrightarrow\Lambda)\in\CM_{\cO,\cS}^{\cM}(R)$ where $\dN\in\Pic_{\cO}^{\cS}(\cM)$ and $\Lambda=\End_{\cO}^{\cS}(\dN)$. Pick a representant $J$ of a class $[J]\in\Pic(R)$. Then $\psi(J^{-1})\Lambda$ is a locally free rank-1 right $\Lambda$-module.
Write $[J]\ast \dN:=\psi(J^{-1})\Lambda\otimes_{\Lambda}\dN\in\Pic_{\cO}^{\cS}(\cM)$ for the element in $\Pic_\cO^\cS(\cM)$ corresponding to it by the correspondence of  Theorem \ref{1-1corrbim}. Note that this construction does not depend on the representant $J$, since $[J]\ast \dN$ only depends on the isomorphism class of the rank-1 right $\Lambda$-module $\psi(J^{-1})\Lambda$.
Since $R$ acts on $\psi(J^{-1})\Lambda$, there is an action of $R$ on $\psi(J^{-1})\Lambda\otimes_{\Lambda}\dN$ which commutes with the actions of both $\cO$ and $\cS$.
This yields a natural embedding $[J]\ast\psi:R\hookrightarrow\End_\cO^\cS([J]\ast \dN)$, which is optimal because
\[\{x\in K: \;\psi(x)(\psi(J^{-1})\Lambda\otimes_{\Lambda}\cM)\subseteq\psi(J^{-1})\Lambda\otimes_{\Lambda}\dN\}=R,\]
and does not depend on the representant $J$ of $[J]$. Hence, it defines an action of $[J]\in\Pic(R)$ on $(\dN,\psi)\in\CM_{\cO,\cS}^\cM(R)$. Namely $[J]\ast(\dN,\psi)=([J]\ast \dN,[J]\ast\psi)\in\CM_{\cO,\cS}^\cM(R)$.

Given that both sets $\CM_{D,N}(R)$ and $\CM_{\cO,\cS}^{\cM}(R)$ are equipped with an action of $\Pic(R)$, it seems reasonable to ask about the behaviour of the action of $\Pic(R)$ under the map $\phi:\CM_{D,N}(R)\ra\coprod_\cM\CM_{\cO,\cS}^{\cM}(R)$ of \eqref{Phi}. This is the aim of the rest of this subsection.

Recall that, $B$ being indefinite, the orders $\cO$ and $I\ast\cO$ are isomorphic for any locally free left $\cO$-module $I$ of rank 1.
\begin{lemma}\label{lemdifbimod}
Let $\cM$ be a $(\cO,\cS)$-bimodule and let $I$ be a locally free left-$\cO$-module of rank 1. Then $I^{-1}\otimes_{\cO }\cM$ admits a structure of $(I\ast\cO,\cS)$-bimodule and the isomorphism $\cO\cong I\ast\cO$ identifies the $(\cO,\cS)$-bimodule $\cM$ with the $(I\ast\cO,\cS)$-bimodule $I^{-1}\otimes_{\cO }\cM$.
\end{lemma}
\begin{proof}
Since $B$ is indefinite, $\Pic(D,N)=1$ and $I$ must be principal. Write $I=\cO\gamma$. Then $I\ast\cO=\gamma^{-1}\cO\gamma$ and the isomorphism $I\ast\cO\cong\cO$ is given by $\gamma^{-1}\alpha\gamma\mapsto\alpha$. Finally, the isomorphism of $\Z$-modules
\[
\begin{array}{ccc}
I^{-1}\otimes_{\cO }\cM & \longrightarrow & \cM\\
 \gamma^{-1}\beta\otimes m &\longmapsto & \beta m
\end{array}
\]
is compatible with the isomorphism $I\ast\cO\cong\cO$ described above.
\end{proof}

\begin{theorem}\label{actPic}
The map $\phi:\CM_{D,N}(R)\longrightarrow \bigsqcup_\cM\CM_{\cO,\cS}^{\cM}(R)$ satisfies the reciprocity law
\[\phi([J]\ast\varphi)=[J]^{-1}\ast\phi(\varphi),\]
for any $\varphi:R\hookrightarrow\cO$ in $\CM_{D,N}(R)$ and any $[J]\in\Pic(R)$.
\end{theorem}
\begin{proof}
The map $\phi$ is given by
\[
\begin{array}{rclc}
\phi: & \CM_{D,N}(R) & \longrightarrow & \bigsqcup_\cM\CM_{\cO,\cS}^{\cM}(R)\\
 & (\varphi:R\hookrightarrow\cO) & \longmapsto & (\cO\otimes_R\cS,\phi_\varphi:R\hookrightarrow\End_{\cO}^{\cS}(\cO\otimes_R\cS))
\end{array}
\]
Let $\varphi:R\hookrightarrow\cO$ denote the conjugacy class of an embedding in $\CM_{D,N}(R)$ and let $[J]\in\Pic(R)$. Write $[J]\ast\varphi:R\hookrightarrow [J]\ast\cO$ for the embedding induced by the action of $[J]$ on $\varphi$.
By Lemma \ref{auxactPic},
\[
([J]\ast\cO)\otimes_R\cS=(\varphi(J^{-1})\cO\otimes_{\cO}\cO\varphi(J))\otimes_{R}\cS.
\]
Then Lemma \ref{lemdifbimod} asserts that, under the isomorphism $[J]\ast\cO\cong \cO$, the $([J]\ast\cO,\cS)$-bimodule $(\varphi(J^{-1})\cO\otimes_{\cO}\cO\varphi(J))\otimes_{R}\cS$ is naturally identified with the $(\cO,\cS)$-bimodule $\cO\varphi(J)\otimes_{R}\cS$.
Moreover, the embedding $\phi_{[J]\ast\varphi}$ is given by
\[
\begin{array}{ccc}
R & \hookrightarrow & \End_{\cO}^{\cS}(\cO\varphi(J)\otimes_{R}\cS)\\
\delta & \longmapsto & (\alpha\varphi(j)\otimes s\mapsto\alpha\varphi(j\delta)\otimes s).
\end{array}
\]
Setting $\Lambda=\End_{\cO}^{\cS}(\cO\otimes_R\cS)$, we easily obtain that
\[
\cO\varphi(J)\otimes_{R}\cS=\phi(\varphi)(J)\Lambda\otimes_{\Lambda}(\cO\otimes_R\cS)=[J]^{-1}\ast(\cO\otimes_R\cS).
\]
Finally, since the action of $\phi_{[J]\ast\varphi}(R)$ on $[J]^{-1}\ast(\cO\otimes_R\cS)$ is given by the natural action of $R$ on $J$, we conclude that $\phi([J]\ast\varphi)=[J]^{-1}\ast\phi(\varphi)$.

\end{proof}

Let $\cM$ be an admissible $(\cO,\cS)$-bimodule of rank 8. By Theorem \ref{1-1bimodref}, the map $\cM\longmapsto\End_{\cO}^{\cS}(\cM)$ induces a one-to-one correspondence between the sets $\Pic_\cO^\cS(\cM)$ and $\Pic(D_0^\cM,NN_0^\cM)$, where $D_0^\cM$ and $N_0^\cM$ were already defined in \S\ref{RibWork}.
This implies that the set $\CM_{\cO,\cS}^\cM(R)$ can be identified with $\CM_{D_0^\cM,N_0^\cM N}(R)$, under the above correspondence.

Both sets are endowed with an action of the group $\Pic(R)$. We claim that the bijection
\begin{equation}\label{admCM}
\CM_{\cO,\cS}^\cM(R)\simeq\CM_{D_0^\cM,N_0^\cM N}(R)
\end{equation}
is equivariant under this action.

Indeed for any $[J]\in\Pic(R)$ and any $(\dN,\psi:R\hookrightarrow\End_\cO^\cS(\dN))$ in $\CM_{\cO,\cS}^\cM(R)$, Theorem \ref{1-1corrbim} asserts that $\Hom_\cO^\cS(\dN,[J]\ast \dN)=\psi(J^{-1})\End_\cO^\cS(\dN)$. Therefore
\[
\End_{\cO}^{\cS}([J]\ast \dN)=\{\rho\in\End_B^H(\dN^0):\;\rho\Hom_{\cO}^{\cS}(\dN,[J]\ast \dN)\subseteq\Hom_{\cO}^{\cS}(\dN,[J]\ast \dN)\}=[J]\ast\End_{\cO}^{\cS}(\dN),
\]
which is the left Eichler order of $\psi(J^{-1})\End_\cO^\cS(\dN)$. Moreover, $[J]\ast\psi$ arises from the action of $R$ on $\psi(J^{-1})\End_\cO^\cS(\dN)$ via $\psi$. In conclusion, both actions coincide.

\begin{corollary}\label{CorPic}
Assume that for all $P\in\CM(R)$ the bimodules $\cM_{\widetilde P}$ are admissible.
Then the map $\phi:\CM_{D,N}(R)  \ra  \bigsqcup_\cM\CM_{D_0^\cM,NN_0^\cM}(R)$, satisfies $\phi([J]\ast\varphi)=[J]^{-1}\ast\phi(\varphi)$,
for any $\varphi:R\hookrightarrow\cO$ in $\CM_{D,N}(R)$ and any $[J]\in\Pic(R)$.
\end{corollary}

\subsection{Atkin-Lehner involutions}\label{ALss}

Recall from \S\ref{QOOE} that the set of optimal embeddings $\CM_{D,N}(R)$ is also equipped with an action of the group $W(D,N)$ of Atkin-Lehner involutions. Let $q\mid DN,\;q\neq p$ be a prime and let $n\geq 1$ be such that $q^n\parallel DN$.
For any $(\cO,\cS)$-bimodule $\cM$, there is a natural action of $w_{q^n}$ on $\CM_{\cO,\cS}^\cM(R)$ as well, as we now describe.

Let $\dQ_\cO$ be the single two-sided $\cO$-ideal of norm $q^n$. Let $(\dN,\psi:R\hookrightarrow \End_\cO^\cS(\dN))\in\CM_{\cO,\cS}^\cM(R)$.
Since $\dQ_\cO$ is two-sided, $\dQ_\cO\otimes_\cO \dN$ acquires a natural structure of $(\cO,\cS)$-bimodule.

Recall that $\cO$ equals $\dQ_\cO\ast\cO$ as orders in $B$. Hence the algebra $\End_\cO^\cS(\dQ_\cO\otimes_\cO \dN)$ is isomorphic to $\End_{\dQ_\cO\ast\cO}^\cS(\dQ_\cO\otimes_\cO \dN)=\End_\cO^\cS(\dN)$ by Lemma \ref{lemdifbimod}.

Moreover, the bimodule $\dQ_\cO\otimes_\cO \dN$ is locally isomorphic to $\dN$ at all places of $\Q$ except possibly for $q$. Since we assumed $q\neq p$, there is a single isomorphism class locally at $q$. Hence we deduce that $\dQ_\cO\otimes_\cO \dN\in\Pic_\cO^\cS(\cM)$. Thus $w_{q^n}$ defines an involution on $\CM_{\cO,\cS}^\cM(R)$ by the rule $w_{q^n}(\dN,\psi)=(\dQ_\cO\otimes_\cO \dN,\psi)$.

We now proceed to describe the behavior of the Atkin-Lehner involution $w_{q^n}$ under the map $$\phi:\CM_{D,N}(R)  \ra  \bigsqcup_\cM\CM_{\cO,\cS}^\cM(R)$$ introduced above.


\begin{theorem}\label{actA_L}
For all $\varphi:R\hookrightarrow\cO$ in $\CM_{D,N}(R)$,
\[\phi(w_{q^n}(\varphi))=w_{q^n}(\phi(\varphi)).\]
\end{theorem}
\begin{proof}
For any $\varphi\in\CM_{D,N}(R)$, let $w_{q^n}(\varphi):R\hookrightarrow \dQ_\cO\ast\cO$ as in \S\ref{QOOE}. Set $\phi(\varphi)=(\cO\otimes_R\cS,\phi_\varphi)$.

As ring monomorphisms, $\varphi$ equals $w_{q^n}(\varphi)$. Hence, $\phi(w_{q^n}(\varphi))=((\dQ_\cO\ast\cO)\otimes_R\cS,\phi_\varphi)$. Applying Lemma \ref{auxactPic}, we obtain that
\[
(\dQ_\cO\ast\cO)\otimes_R\cS=(\dQ_\cO^{-1}\otimes_{\cO}\dQ_\cO)\otimes_{R}\cS.
\]
By Lemma \ref{lemdifbimod}, the $(\dQ_\cO\ast\cO,\cS)$-bimodule $(\dQ_\cO^{-1}\otimes_{\cO}\dQ_\cO)\otimes_{R}\cS)$ corresponds to the $(\cO,\cS)$-bimodule $\dQ_\cO\otimes_{R}\cS=\dQ_\cO\otimes_\cO(\cO\otimes_{R}\cS)$. Thus we conclude that
$\phi(w_{q^n}(\varphi))=(\dQ_\cO\otimes_\cO(\cO\otimes_{R}\cS),\phi_\varphi)=w_{q^n}(\phi(\varphi))$.
\end{proof}

\begin{remark}\label{A_L_remark}
We defined an action of $w_{q^n}$ on $\CM_{\cO,\cS}^\cM(R)$ for any $(\cO,\cS)$-bimodule $\cM$. Returning to the situation where $\cM$ is admissible, one can ask if, through the correspondence $\CM_{\cO,\cS}^\cM(R)\leftrightarrow\CM_{D_0^\cM,N_0^\cM N}(R)$ of \eqref{admCM}, this action agrees with the Atkin-Lehner action $w_{q^n}$ on $\CM_{D_0^\cM,N_0^\cM N}(R)$.

Indeed, let $(\dN,\psi)\in\CM_{\cO,\cS}^\cM(R)$ and set $\Lambda=\End_\cO^\cS(\dN)\in\Pic(D_0^\cM,N_0^\cM N)$. Since $w_{q^n}(\dN,\psi)=(\dQ_\cO\otimes_\cO \dN,\psi)$, we only have to check that $\dQ_\Lambda:=\Hom_\cO^\cS(\dN,\dQ_\cO\otimes_\cO \dN)$ is a locally free rank-1 two-sided $\Lambda$-module of norm $q^n$ in order to ensure that
$\End_\cO^\cS(\dQ_\cO\otimes_\cO \dN)=\dQ_\Lambda\ast\Lambda$.

Since $\dQ_\Lambda$ is naturally a right $\End_\cO^\cS(\dQ_\cO\otimes_\cO \dN)$-module and $\Lambda$ equals $\End_\cO^\cS(\dQ_\cO\otimes_\cO \dN)$ as orders in $\End_B^H(\dN^0)$, we conclude that $\dQ_\Lambda$ is two-sided. In order to check that $\dQ_\Lambda$ has norm $q^n$, note that
$\dQ_\Lambda$ coincides with $\Hom_\cO^\cS(\dQ_\cO\otimes_\cO \dN,(\dQ_\cO\otimes_\cO\dQ_\cO)\otimes_\cO \dN)$ as ideals on the ring $\Lambda$. Hence
\[
\dQ_\Lambda^2= \Hom_\cO^\cS(\dQ_\cO\otimes_\cO \dN,q^n \dN)\cdot\Hom_\cO^\cS( \dN,\dQ_\cO\otimes_\cO \dN)=\Hom_\cO^\cS( \dN,q^n \dN)=q^n\Lambda.
\]
\end{remark}

\section{\v{C}erednik-Drinfeld's special fiber}\label{CD}

In this section we exploit the results of \S\ref{ssesp} to describe the specialization of Heegner points on Shimura curves $X_0(D,N)$ at primes $p\mid D$. In order to do so, we first recall basic facts about the moduli interpretation of \v{C}erednik-Drinfeld's special fiber of $X_0(D,N)$ at $p$. 

Let $p$ be a prime dividing $D$, fix $\F$ an algebraic closure of $\F_p$ and let $\tilde X_0(D,N)=\cX_0(D,N)\times\Spec(\F_p)$. By the work of \v{C}erednik and Drinfeld, all irreducible components of $\tilde X_0(D,N)$ are reduced smooth conics, meeting transversally at double ordinary points.

Points in $\tilde X_0(D,N)(\F)$ parameterize abelian surfaces $(\widetilde A,\widetilde i)$ over $\F$ with QM by $\cO$ such that $\mbox{Trace}_{\F }(\widetilde i(\alpha)\mid \mbox{Lie}(\widetilde A))=\Trace(\alpha)\in \Q$ for all $\alpha\in \cO$. Here $\Trace$ stands for the reduced trace on $\cO$.


It follows from \cite[Lemma 4.1]{Rib} that any abelian surface with QM by $\cO$ in characteristic $p\mid D$ is supersingular.
Moreover, singular points of $\tilde X_0(D,N)$ correspond to abelian surfaces $(\widetilde A,\widetilde i)$ with QM by $\cO$ such that $a(\widetilde A)=2$ and their corresponding bimodule is admissible at $p$, of type $(1,1)$ (c.f \cite[Section 4]{Rib}). Observe that $p$ is the single place at which both $\cO$ and $\cS$ ramify. Hence for such bimodules $\cM$ we have $\Sigma=\{p\}$, $D_0^\cM=D/p$ and $N_0^\cM=p$ in the notation of \S\ref{Admbim}.

\begin{remark}\label{remtrp-pair}
In Ribet's original paper \cite{Rib}, singular points are characterized as triples $[A_0,i_0,C]$ where $(A_0,i_0)$ is an abelian surface with QM by a maximal order and $C$ is a $\Gamma_0(N)$-structure. According to Appendix A, we can construct from such a triple a pair $(A,i)$ with QM by $\cO$ such that $\End(A,i)=\End(A_0,i_0,C)$.
\end{remark}

Let $P=[A,i]\in\CM(R)$ be a Heegner point and let $\varphi_P\in\CM_{D,N}(R)$ be the optimal embedding attached to $P$. It is well know that if such an optimal embedding exists, $R$ is maximal at $p$ and $p$ either ramifies or is inert in $K$ (cf. \cite[Theorem 3.1]{Vig}). Thus, $[\widetilde A,\widetilde i]$ is supersingular in $\tilde X_0(D,N)$ and the bimodule associated to $\widetilde P$ is $\cM_{\widetilde P}=\cO\otimes_R\cS$, by Theorem \ref{CMbimod}. Moreover, by Remark \ref{remlemaopt} the embedding $\phi_P$ is optimal.
\begin{proposition}\label{Howbimis}
Assume that $p$ ramifies in $K$. Then $\cM_{\widetilde  P}$ is admissible at $p$ and $(\cM_{\widetilde  P})_p$ is of type $(1,1)$. Furthermore, the algebra $\End_{\cO}^{\cS}(\cM_{\widetilde  P})$ admits a natural structure of oriented Eichler order of level $Np$ in the quaternion algebra of discriminant $\frac{D}{p}$.
\end{proposition}
\begin{proof}
Since $\cO_p=\cS_p$ is free as right $R_p$-module, we may choose a basis of $\cO_p$ over $R_p$. In terms of this basis, the action of $\cO_p$ on itself given by right multiplication is described by a homomorphism
\[f:\cO_p\hookrightarrow M(2,R_p).\]
Since $p$ ramifies in $K$, the maximal ideal $\mathfrak{p}_{\cO_p}$ of $\cO_p$ is generated by an uniformizer $\pi$ of $R_p$ (cf. \cite[Corollaire II.1.7]{Vig}). This shows that $f(\mathfrak{p}_{\cO_p})\subseteq M(2,\pi R_p)$.

This allows us to conclude that $\cM_p$ is admissible at $p$. Indeed, in matrix terms, the local bimodule $\cM_p$ is given by the composition of $f$ with the natural inclusion $M(2,R_p)\hookrightarrow M(2,\cS_p)$. Notice that $M(2,\pi R_p)$ is mapped into $M(2,\mathfrak{p}_{\cS_p})$ under that inclusion.

In order to check that $r_p=1$, reduction modulo $\mathfrak{p}_{\cO_p}$ yields an embedding
\[\widetilde{f}:\cO/\mathfrak{p}_{\cO}\hookrightarrow M(2,\F_p)\]
because the residue field of $K$ at $p$ is the prime field $\F_p$. 
After extending to a quadratic extension of $\F_p$, this representation of $\cO/\mathfrak{p}_{\cO}\cong\F_{p^2}$ necessarily splits into the direct sum of the two embeddings of $\F_{p^2}$ into $\F$. On the other hand, we know that $\cM_p=\cO_p^r\times\mathfrak{p}_{\cO_p}^s$ with $s+r=2$. Consequently, we deduce that $\cM_p\cong\cO_p\times\mathfrak{p}_{\cO_p}$ and $\cM_p$ is of type $(1,1)$.

Finally, since in this case $D_0=\frac{D}{p}$ and $N_0=p$, it follows from Theorem \ref{1-1bimodref} that $\End_{\cO}^{\cS}(\cM)\in\E(D/p,Np)$ .
\end{proof}

\begin{theorem}\label{redCMsing}
A Heegner point $P\in \CM(R)$ of $X_0(D,N)$ reduces to a singular point of $\tilde X_0(D,N)$ if and only if $p$ ramifies in $K$.
\end{theorem}
\begin{proof}
As remarked above, $p$ is inert or ramifies in $K$. Assume first that $p$ ramifies in $K$. Then by Proposition \ref{Howbimis} the bimodule $\cM_{\widetilde P}$ is admissible at $p$ and of type $(1,1)$. It follows from \cite[Theorem 4.15, Theorem 5.3]{Rib} that the point $\widetilde P\in\tilde X_0(D,N)$ is singular.

Suppose now that $\widetilde P=[\widetilde A,\widetilde i]$ is singular.
Then its corresponding $(\cO,\cS)$-bimodule $\cM_{\widetilde P}$ is admissible at $p$ and of type $(1,1)$. Thus $\End(\widetilde A,\widetilde i)=\End_\cO^\cS(\cM_{\widetilde P})\in\Pic(D/p,Np)$ and the conjugacy class of the optimal embedding $\phi_P:R\hookrightarrow \End_\cO^\cS(\cM_{\widetilde P})$ is an element of $\CM_{\frac{D}{p},Np}(R)$. If such an embedding exists, then $p$ can not be inert in $K$ thanks to \cite[Theorem 3.2]{Vig}. Hence $p$ ramifies in $K$.
\end{proof}

That points $P\in\CM(R)$ specialize to the non-singular locus of $\tilde X_0(D,N)$ when $p$ is unramified in $K$ was already known by the experts (cf. e.g. \cite{Longo}) and can also be easily deduced by rigid analytic methods. The novelty of Theorem \ref{redCMsing} is the converse.

\subsection{Heegner points and the singular locus}\label{CMsing}

As explained before, singular points of $\tilde X_0(D,N)$ are in one-to-one correspondence with isomorphism classes of $(\cO,\cS)$-bimodules which are admissible at $p$ and of type $(1,1)$. By Theorem \ref{1-1bimodref}, such an isomorphism class is characterized by the isomorphism class of the endomorphism ring $\End_\cO^\cS(\cM)$, which is an oriented Eichler order of level $Np$ in the quaternion algebra of discriminant $D/p$. Thus the set $\tilde X_0(D,N)_{\rm{sing}}$ of singular points of $\tilde X_0(D,N)$ is in natural one-to-one correspondence with $\Pic(D/p,Np)$.

Let $P=[A,i]\in\CM(R)$ be a Heegner point. 
As proved in Theorem \ref{redCMsing}, $P$ specializes to a singular point $\widetilde P$ if and only if $p$ ramifies in $K$. If this is the case, the optimal embedding $\phi_P:R\hookrightarrow\End_\cO^\cS(\cM_{\widetilde P})$ provides an element of $\CM_{\frac{D}{p},Np}(R)$.

Therefore, the map $\phi$ of \eqref{Phi}, which was constructed by means of the reduction of $X_0(D,N)$ modulo $p$, can be interpreted as a map between CM-sets:
\begin{equation}\label{Phi1}
\Phi:\CM_{D,N}(R)\longrightarrow \CM_{\frac{D}{p},Np}(R).
\end{equation}
Moreover, composing with the isomorphism $\CM(R)\simeq \CM_{D,N}(R)$ of $\eqref{phip}$ and the projection $\pi:\CM_{\frac{D}{p},Np}(R)\ra\Pic(\frac{D}{p},Np)$ of \eqref{proj}, the resulting map $\CM(R)\ra\Pic(\frac{D}{p},Np)$ describes the specialization of $P\in\CM(R)$ at $p$.

\begin{theorem}\label{CMsingthm}
The map $\Phi:  \CM_{D,N}(R)  \longrightarrow  \CM_{\frac{D}{p},Np}(R)$ is equivariant for the action of $W(D,N)$ and, up to sign, of $\Pic(R)$. More precisely:
\[\Phi([J]\ast\varphi)=[J]^{-1}\ast\Phi(\varphi),\quad\quad\Phi(w_m(\varphi))=w_m(\Phi(\varphi))\]
for all $m\parallel DN$, $[J]\in\Pic(R)$ and $\varphi:R\hookrightarrow\cO$ in $\CM_{D,N}(R)$. Moreover, if $N$ is square free, $\Phi$ is bijective.
\end{theorem}
\begin{proof}
The statement for $\Pic(R)$ is Corollary \ref{CorPic}. It follows from Theorem \ref{actA_L} and Remark \ref{A_L_remark} that $\Phi(w_m(\varphi))=w_m(\Phi(\varphi))$ for all $m\parallel ND/p$. Since $p$ ramifies in $K$, $w_p\in W(D,N)\cong W(\frac{D}{p},Np)$ preserves local equivalence by \cite[Theorem II.3.1]{Vig}. Hence the action of $w_p$ coincides with the action of some $[\dP]\in\Pic(R)$.
This shows that $\Phi(w_m(\varphi))=w_m(\Phi(\varphi))$ for all $m\parallel ND$.

Finally, in order to check that $\Phi$ is a bijection when $N$ is square free, observe that $\Pic(R)\times W_{D,N}(R)=\Pic(R)\times W_{\frac{D}{p},Np}(R)$ acts freely and transitively both on $\CM_{D,N}(R)$ and on $\CM_{\frac{D}{p},Np}(R)$.
\end{proof}

\subsection{Heegner points and the smooth locus}\label{CMsmooth}

In \cite[\S4]{Rib}, Ribet describes the smooth locus $\tilde X_0(D,N)_{\rm{ns}}$ of $\tilde X_0(D,N)$ in terms of abelian surfaces $(\widetilde A,\widetilde i)$ with QM over $\F$. We proceed to summarize the most important ideas of such description.

Let $(\widetilde A,\widetilde i)$ be an abelian variety with QM by $\cO$. By \cite[Propositions 4.4 and 4.5]{Rib}, the isomorphism class $[\widetilde A,\widetilde i]$ defines a non-singular point $\widetilde P\in \tilde X_0(D,N)_{\rm{ns}}$ if and only if $\widetilde A$ has exactly one subgroup scheme $H_{\widetilde P}$, which is $\cO$-stable and isomorphic to $\alpha_p$. There $\alpha_p$ stands for the usual inseparable group scheme of rank $p$. Furthermore, by considering the quotient $\widetilde B=\widetilde A/H_{\widetilde P}$ and the embedding $j:\cO\hookrightarrow\End_\F \widetilde B$ induced by $\widetilde i$, we obtain an abelian surface $(\widetilde B,j)$ with QM such that $a(\widetilde B)=2$. The pair $(\widetilde B,j)$ defines an admissible bimodule $\cM^{\widetilde P}=\cM_{(\widetilde B,j)}$ which has either type $(2,0)$ or type $(0,2)$.

The set of irreducible components of $\tilde X_0(D,N)$ is in one-to-one correspondence with the set of isomorphism classes of admissible $(\cO,\cS)$-bimodules of type $(2,0)$ and $(0,2)$. With this in mind, the bimodule $\cM^{\widetilde P}$ determines the component where the point $\widetilde P$ lies. Moreover, the Atkin-Lehner involution $w_p\in\End_\Q(X_0(D,N))$ maps bimodules of type $(2,0)$ to those of type $(0,2)$, and viceversa. By Theorem \ref{1-1bimodref},
such bimodules $\cM_{(\widetilde B,j)}$ are characterized by their type and their endomorphism ring $\End_\cO^\cS(\cM_{(\widetilde B,j)})=\End(\widetilde B,j)\in\Pic(\frac{D}{p},N)$. Hence the set of
the irreducible components of $\tilde X_0(D,N)$ is in one-to-one correspondence with two copies of $\Pic(\frac{D}{p},N)$, one copy for each type, $(0,2)$ and $(2,0)$. 
Finally, the automorphism $w_p$ exchanges both copies of $\Pic(\frac{D}{p},N)$.

Let $P=[\widetilde A,\widetilde i]\in\tilde X_0(D,N)$ be a non-singular point and assume that $a(\widetilde A)=2$. Let $\cM_{\widetilde P}$ be its associated $(\cO,\cS)$-bimodule. The subgroup scheme $H_{\widetilde P}$ gives rise to a degree-$p$ isogeny $\mu_{\widetilde P}:\widetilde A\ra\widetilde B=\widetilde A/H_{\widetilde P}$ such that $\mu_{\widetilde P}\widetilde i(\alpha)=j(\alpha)\mu_{\widetilde P}$ for all $\alpha\in\cO$.

Since $\widetilde{A}\simeq\widetilde B\simeq \widetilde E^2$, we may fix an isomorphism of algebras $\End_\F(\widetilde{A})\ra M_2(\cS)$. Then each isogeny can be regarded as a matrix with coefficients in $\cS$. In order to characterize the bimodule $\cM^{\widetilde P}$ in terms of $\cM_{\widetilde P}$ we shall use the following proposition.

\begin{proposition}\label{bim-bim}
Let $(\mathcal A, \mathfrak{i})/\F$ and $(\mathcal A,\mathfrak{j})/\F$ be abelian surfaces with QM by $\cO$ such that $a(\mathcal A)=2$. Let $\cM$ and $\widehat{\cM }$ be their associated $(\cO,\cS)$-bimodules. Consider the usual morphisms attached to each bimodule
\[f_{\cM }:\cO\longrightarrow M_2(\cS),\quad f_{\widehat{\cM} }:\cO\longrightarrow \cM_2(\cS),\]
and assume there exists an isogeny $\gamma:\mathcal A\longrightarrow \mathcal A$ such that $f_{\cM}(\alpha)\gamma=\gamma f_{\widehat{\cM} }(\alpha)$ for all $\alpha\in\cO$.
Then the image $\gamma \cM$, where $\cM$ is viewed as a free right $\cS$-module of rank 2, is $\cO$-stable.
Furthermore $\widehat{\cM}\cong\gamma \cM$ as $(\cO,\cS)$-bimodules.
\end{proposition}
\begin{proof}
The right $\cS$-module $\gamma \cM$ is $\cS$-free of rank 2 with basis $\{\gamma e_1,\gamma e_2\}$, provided that $\{e_1,e_2\}$ is a $\cS$-basis for $\cM$.
An element of $\gamma \cM$ can be writen as:
\[\gamma e_1a+\gamma e_2 b=\left(\begin{array}{cc}\gamma e_1 &\gamma  e_2\end{array}\right)\left(\begin{array}{c} a \\  b\end{array}\right)=\left(\begin{array}{cc} e_1 &  e_2\end{array}\right)\gamma\left(\begin{array}{c} a \\  b\end{array}\right).\]
Therefore any $\alpha\in\cO$ acts on it as follows:
\[
\left(\begin{array}{cc} e_1 &  e_2\end{array}\right)f_{\cM }(\alpha)\gamma\left(\begin{array}{c} a \\  b\end{array}\right)=\left(\begin{array}{cc} e_1 &  e_2\end{array}\right)\gamma f_{\widehat{\cM}}(\alpha)\left(\begin{array}{c} a \\  b\end{array}\right)=\left(\begin{array}{cc}\gamma e_1 & \gamma e_2\end{array}\right)f_{\widehat{\cM}}(\alpha)\left(\begin{array}{c} a \\  b\end{array}\right).
\]
Thus $\gamma \cM$ is $\cO$-stable and $\cO$ acts on it through the map $f_{\widehat{\cM}}$. We conclude that $\gamma \cM\cong\widehat{\cM}$ as $(\cO,\cS)$-bimodules.
\end{proof}
Applying the above proposition to $f_{\cM }=j$, $f_{\widehat{\cM}}=\widetilde i$ and $\gamma=\mu_{\widetilde P}$, we obtain that $\mu_{\widetilde P}\cM^{\widetilde P}=\cM_{\widetilde P}$.


Note that endomorphisms $\lambda\in\End(\widetilde A)$ which fix $H_{\widetilde P}$ give rise to endomorphisms $\widehat\lambda\in\End(\widetilde B)$. If in addition $\lambda\in\End(\widetilde A, \widetilde i)$ then $\widehat\lambda$ lies in $\End(\widetilde B,j)$.

\begin{lemma}
Every endomorphism in $\End(\widetilde A, \widetilde i)$ fixes $H_{\widetilde P}$.
\end{lemma}
\begin{proof}
Let $\lambda\in\End(\widetilde A, \widetilde i)$. Then $\lambda(H_{\widetilde P})$ is either trivial or a subgroup scheme of rank $p$. If $\lambda(H_{\widetilde P})=0$, the statement follows. Assume thus that $\lambda(H_{\widetilde P})$ is a subgroup scheme of rank $p$. Since $\widetilde A$ is supersingular, $\lambda(H_{\widetilde P})$ is isomorphic to $\alpha_p$. Moreover, for all $\alpha\in\cO$ we have that $\widetilde i(\alpha)(\lambda(H_{\widetilde P}))=\lambda(\widetilde i(\alpha)(H_{\widetilde P}))=\lambda(H_{\widetilde P})$. Therefore $\lambda(H_{\widetilde P})$ is $\cO$-stable and consequently $\lambda(H_{\widetilde P})=H_{\widetilde P}$, by uniqueness.
\end{proof}

By the preceding result, there is a monomorphism $\End(\widetilde A, \widetilde i)\ra\End(\widetilde B, j)$ corresponding, via bimodules, to the monomorphism $\delta_{\widetilde P}:\End_\cO^\cS(\cM_{\widetilde P})\ra\End_\cO^\cS(\cM^{\widetilde P})$ that maps every $\lambda\in\End_\cO^\cS(\cM_{\widetilde P})$ to its single extension to $\cM^{\widetilde P}\supset \cM_{\widetilde P}$.

Let $(\dN,\psi)\in\CM_{\cO,\cS}^{\cM_{\widetilde P}}(R)$. By \S\ref{ss-bim}, $\dN=\cM_{(\widetilde A',\widetilde i')}$ for some abelian surface $(\widetilde A',\widetilde i')$ with QM by $\cO$.
By \cite[Theorem 5.3]{Rib}, whether $[\widetilde A',\widetilde i']$ is a singular point of $\tilde X_0(D,N)$ or not depends on the local isomorphism class of $\cM_{(\widetilde A',\widetilde i')}$ at $p$.
Since $\cM_{(\widetilde A',\widetilde i')}$ and $\cM_{\widetilde P}$ are locally isomorphic, $[\widetilde A',\widetilde i']=\widetilde Q\in\tilde X_0(D,N)_{\rm{ns}}$ and we may write $\cM_{(\widetilde A',\widetilde i')}=\cM_{\widetilde Q}$.
The composition of $\psi$ with $\delta_{\widetilde Q}$ gives rise to an embedding $R\hookrightarrow \End_\cO^\cS(\cM^{\widetilde Q})$, which we claim is optimal.

Indeed, since $\mu_{\widetilde Q}$ has degree $p$, $[\cM^{\widetilde Q}:\cM_{\widetilde Q}]=p$. Hence the inclusion $\End_\cO^\cS(\cM_{\widetilde Q})\subset\End_\cO^\cS(\cM^{\widetilde Q})$ has also $p$-power index. Due to the fact that $\psi$ is optimal,
\[
\delta_{\widetilde Q}\circ\psi(R)=\delta_{\widetilde Q}(\End_\cO^\cS(\cM_{\widetilde Q})\cap \psi(K))\subseteq\End_\cO^\cS(\cM^{\widetilde Q})\cap \delta_{\widetilde Q}\circ\psi(K)=:\delta_{\widetilde Q}\circ\psi(R'),
\]
where $R'$ is an order in $K$ such that $R\subseteq R'$ has $p$-power index. Recall that $R$ is maximal at $p$, hence we conclude that $R'=R$ and $\delta_{\widetilde Q}\circ\psi$ is optimal.


Since $H_{\widetilde Q}$ is the single $\cO$-stable subgroup scheme of $\widetilde A'$ of rank $p$, $\cM^{\widetilde Q}$ is the single $\cO$-stable extension of $\cM_{\widetilde Q}$ of index $p$. Moreover, it is characterized by the local isomorphism class of $\cM^{\widetilde Q}$ at $p$. More precisely, if $\cM_{\widetilde P}$ and $\cM_{\widetilde Q}$ are locally isomorphic, then $\cM^{\widetilde P}$ and $\cM^{\widetilde Q}$ must be locally isomorphic also.
Thus the correspondence $(\dN=\cM_{\widetilde Q},\psi)\rightarrow (\cM^{\widetilde Q},\delta_{\widetilde Q}\circ\psi)$ induces a map \begin{equation}\label{eqCMraro}
\delta:\CM_{\cO,\cS}^{\cM_{\widetilde P}}(R)\ra\CM_{\cO,\cS}^{\cM^{\widetilde P}}(R).
\end{equation}

Both sets are equipped with an action of $\Pic(R)$ and involutions $w_{q^n}$, for all $q^n\parallel DN$, $q\neq p$. On the other hand, since $\cM^{\widetilde P}$ is admissible of type $(2,0)$ or $(0,2)$, we identify $\CM_{\cO,\cS}^{\cM^{\widetilde P}}(R)$ with $\CM_{D/p,N}(R)$ by Theorem \ref{1-1bimodref} .
\begin{lemma}\label{lemmaPicAL}
The map $\delta:\CM_{\cO,\cS}^{\cM_{\widetilde P}}(R)\ra\CM_{D/p,N}(R)$ is equivariant under the actions of $\Pic(R)$ and $W(D/p,N)$.
\end{lemma}
\begin{proof}
Let $(\cM_{\widetilde Q},\psi)\in\CM_{\cO,\cS}^{\cM_{\widetilde P}}(R)$, let $[J]\in\Pic(R)$ and write $\Lambda=\End_\cO^\cS(\cM_{\widetilde Q})$.
We denote by $\widetilde Q^{[J]}\in\tilde X_0(D,N)$ the point attached to the bimodule $[J]\ast \cM_{\widetilde Q}=\psi(J^{-1})\Lambda\otimes_\Lambda \cM_{\widetilde Q}=\cM_{\widetilde Q^{[J]}}$.

Write $\Lambda'=\End_\cO^\cS(\cM^{\widetilde Q})\supset\Lambda$ and consider the bimodule $[J]\ast \cM^{\widetilde Q}=\psi(J^{-1})\Lambda'\otimes_{\Lambda'}  \cM^{\widetilde Q}\in\Pic_\cO^\cS(\cM^{\widetilde P})$. Then $[J]\ast \cM_{\widetilde Q}\subset[J]\ast \cM^{\widetilde Q}$ is $\cO$-stable of index $p$ and $[J]\ast \cM^{\widetilde Q}=\cM^{\widetilde Q^{[J]}}$ by uniqueness. Since $[J]\ast \psi$ and $[J]\ast(\delta_{\widetilde Q}\circ\psi)$ are given by the action of $R$ on $\cM_{\widetilde Q^{[J]}}$ and $\cM^{\widetilde Q^{[J]}}$ respectively,
\[
\delta([J]\ast(\cM_{\widetilde Q},\psi))=(\cM^{\widetilde Q^{[J]}},\delta_{\widetilde Q^{[J]}}\circ[J]\ast\psi)=([J]\ast \cM^{\widetilde Q},[J]\ast(\delta_{\widetilde Q}\circ\psi))=[J]\ast\delta(\cM_{\widetilde Q},\psi).
\]

As for the Atkin-Lehner involution $w_{q^n}$, let $\dQ_\cO$ be the single two-sided ideal of $\cO$ of norm $q^n$. Again $\dQ_\cO\otimes_\cO \cM^{\widetilde Q}$ is an $\cO$-stable extension of $\dQ_\cO\otimes_\cO \cM_{\widetilde Q}$ of index $p$. Hence $\dQ_\cO\otimes_\cO \cM^{\widetilde Q}=\cM^{w_{q^n}(\widetilde Q)}$, where $w_{q^n}(\widetilde Q)\in\tilde X_0(D,N)$ is the non-singular point attached to $\dQ_\cO\otimes_\cO \cM_{\widetilde Q}$.
This concludes that $\delta\circ w_{q^n}=w_{q^n}\circ\delta$.
\end{proof}

Let $P=[A,i]\in\CM(R)$ be a Heegner point such that $\tilde P\in\tilde X_0(D,N)_{\rm{ns}}$. Let $\cM_{\widetilde P}$ the bimodule attached to its specialization $\widetilde P=[\widetilde A,\widetilde i]\in\tilde X_0(D,N)$.
Recall the map $\phi:\CM_{D,N}(R)\ra\bigsqcup_\cM\CM_{\cO,\cS}^{\cM}(R)$ of \eqref{Phi}, induced by the natural injection $\End(A,i)\hookrightarrow\End(\widetilde A,\widetilde i)$ and the correspondences of \eqref{phip} and \eqref{end}.

The aim of the rest of this section is to modify $\phi$ in order to obtain a map
\[
\Phi:\CM_{D,N}(R)\ra\CM_{\frac{D}{p},N}(R)\sqcup\CM_{\frac{D}{p},N}(R)
\]
which composed with the natural projection $\pi:\CM_{\frac{D}{p},N}(R)\sqcup\CM_{\frac{D}{p},N}(R)\ra \Pic(\frac{D}{p},N)\sqcup\Pic(\frac{D}{p},N)$ and the correspondence of $\eqref{phip}$, assigns to $P\in\CM(R)$ the Eichler order that describes the irreducible component at which $P$ specializes.


\begin{theorem}\label{CMsmooththm}
The embedding $\End(A,i)\hookrightarrow\End(\widetilde A, \widetilde i)\hookrightarrow\End(\widetilde B, j)\in\Pic(\frac{D}{p},N)$ induces a map
\begin{equation}\label{Phi2}
\Phi: \CM_{D,N}(R)\ra\CM_{\frac{D}{p},N}(R)\sqcup\CM_{\frac{D}{p},N}(R),
\end{equation}
which is equivariant for the action of $W_{\frac{D}{p},N}(R)$ and satisfies the reciprocity law $\Phi([J]\ast\varphi)=[J]^{-1}\ast\Phi(\varphi)$ for all $[J]\in\Pic(R)$ and $\varphi\in\CM_{D,N}(R)$.

If $N$ is square free, then $\Phi$ is bijective, i.e. it establishes a bijection of $W_{\frac{D}{p},N}(R)$-sets between $\CM_{D,N}(R)/w_p$ and $\CM_{\frac{D}{p},N}(R)$.
\end{theorem}
Notice that we are considering $W(\frac{D}{p},N)$ as a subgroup of $W(D,N)$, so that $W(\frac{D}{p},N)$ acts naturally on $\CM_{D,N}(R)$. Since points in $\CM(R)$ have good reduction, $p$ is inert in $K$ by Theorem \ref{redCMsing}. Therefore, by \cite[Theorem 3.1]{Vig}, $m_p=2$ and the subgroups $W_{D,N}(R)/\langle w_p\rangle$ and $ W_{\frac{D}{p},N}(R)$ are isomorphic. Hence we can consider $\CM_{D,N}(R)/w_p$ as a $W_{\frac{D}{p},N}(R)$-set.
\begin{proof}
The map $\Phi: \CM_{D,N}(R)\ra\CM_{\frac{D}{p},N}(R)\sqcup\CM_{\frac{D}{p},N}(R)$ arises as the composition
\[
\CM_{D,N}(R)\stackrel{\phi}{\longrightarrow}\sqcup_\cM \CM_{\cO,\cS}^\cM(R)\stackrel{\delta}{\longrightarrow}\CM_{D/p,N}(R)\sqcup\CM_{D/p,N}(R).
\]
Then Lemma \ref{lemmaPicAL}, Theorem \ref{actPic} and Theorem \ref{actA_L} prove the first statement.

Assume that $N$ is square free. Then freeness and transitivity of the action of $W_{D/p,N}(R)\times \Pic(R)$ on both $\CM_{D,N}(R)/w_p$ and $\CM_{\frac{D}{p},N}(R)$ show that the corresponding map between them is bijective. Since the Atkin-Lehner involution $w_p$ exchanges the two components of $\CM_{D/p,N}(R)\sqcup\CM_{D/p,N}(R)$, we conclude that $\Phi:\CM_{D,N}(R)\ra\CM_{D/p,N}(R)\sqcup\CM_{D/p,N}(R)$ is also bijective.

\end{proof}

\section{Supersingular good reduction}\label{ssgood}

Exploiting the results of \S\ref{ssesp}, we can also describe the supersingular reduction of Heegner points at primes $p$ of good reduction of the Shimura curve $X_0(D,N)$. Indeed, let $p\nmid DN$ be a prime, let $\F$ be an algebraic closure of $\F_p$ and let $\tilde X_0(D,N)=\cX_0(D,N)\times\F_p$.
If $\widetilde P=[\widetilde A,\widetilde i]\in \tilde X_0(D,N)$ is a supersingular abelian surface with QM, then $a(\widetilde A)=2$ by \cite[\S3]{Rib}. Therefore $[\widetilde A,\widetilde i]$ is characterized by the isomorphism class of the $(\cO,\cS)$-bimodule $\cM_{\widetilde P}$ attached to it. By Theorem \ref{1-1bimodref}, these isomorphism classes are in correspondence with the set $\Pic(Dp,N)$.


Let $P=[A,i]\in\CM(R)$ be a Heegner point and assume that the conductor $c$ of $R$ is coprime to $p$. Since $A$ is isomorphic to a product of two elliptic curves with CM by $R$, it follows from the classical work of Deuring that $A$ has supersingular specialization if and only if $p$ does not split in $K$.

If we are in this case, the map $\phi$ of \eqref{Phi} becomes
\begin{equation}\label{Phi3}
\Phi:\CM_{D,N}(R)\ra\CM_{Dp,N}(R)
\end{equation}
by means of the natural identification \eqref{admCM}.
Composing with the natural projection $$\pi:\CM_{Dp,N}(R)\ra\Pic(Dp,N)$$ and the correspondence of \eqref{phip},
one obtains a map $\CM(R) \mapsto \Pic(Dp,N)$ which assigns to $P=[A,i]\in\CM(R)$ the Eichler order $\End(\widetilde A,\widetilde i)$ that describes its supersingular specialization.

\begin{theorem}
The map $\Phi:  \CM_{D,N}(R)  \longrightarrow  \CM_{Dp,N}(R)$ is equivariant for the action of $W(D,N)$ and, up to sign, of $\Pic(R)$. More precisely:
\[\Phi([J]\ast\varphi)=[J]^{-1}\ast\Phi(\varphi),\quad\quad\Phi(w_m(\varphi))=w_m(\Phi(\varphi))\]
for all $m\parallel DN$, $[J]\in\Pic(R)$ and $\varphi:R\hookrightarrow\cO$ in $\CM_{D,N}(R)$.

Assume that $N$ is square free. If $p$ ramifies in $K$, the map $\Phi$ is bijective. If $p$ is inert in $K$, the induced map
\[
\CM_{D,N}(R)\stackrel{\Phi}{\ra}\CM_{Dp,N}(R)\ra\CM_{Dp,N}(R)/w_p
\]
is also bijective.
\end{theorem}
\begin{proof}
The first statement follows directly from Theorem \ref{actPic} and Theorem \ref{actA_L}.
Assume that $N$ is square free. If $p$ ramifies in $K$, we have that $ W_{D,N}(R)= W_{Dp,N}(R)$. Then $\Pic(R)\times  W_{D,N}(R)$ acts simply and transitively on both $\CM_{D,N}(R)$ and $\CM_{D,N}(R)$ and $\Phi$ is bijective. If $p$ is inert in $K$, then $W_{D,N}(R)= W_{Dp,N}(R)/w_p$ and the last assertion holds.
\end{proof}

\section{Deligne-Rapoport's special fiber}\label{DR}

In this section we exploit the results of \S\ref{ssgood} to describe the specialization of Heegner points on Shimura curves $X_0(D,N)$ at primes $p\parallel N$. In order to do so, we first recall basic facts about Deligne-Rapoport's special fiber of $X_0(D,N)$ at $p$. 

Let $p$ be a prime dividing exactly $N$, fix $\F$ an algebraic closure of $\F_p$ and let $\tilde X_0(D,N)=\cX_0(D,N)\times\Spec(\F_p)$.
By the work of Deligne and Rapoport, there are two irreducible components of $\tilde X_0(D,N)$ and they are isomorphic to $\tilde X_0(D,N/p)$. Notice that $\cX_0(D,N/p)$ has good reduction at $p$, hence $\tilde X_0(D,N/p)$ is smooth. 

Let $X_0(D,N)\rightrightarrows X_0(D,N/p)$ be the two degeneracy maps described in Appendix A. They specialize to maps $\tilde X_0(D,N)\rightrightarrows \tilde X_0(D,N/p)$ on the special fibers; we denote them by $\delta$ and $\delta w_p$.
Write $\gamma$ and $w_p \gamma$ for the maps $\tilde X_0(D,N/p)\rightrightarrows \tilde X_0(D,N)$ given in terms of the moduli interpretation of these curves described in Appendix A by $\gamma((\widetilde{A},\widetilde{i}))=(\widetilde{A},\widetilde{i},\ker(\Frob))$ and $w_p\gamma((\widetilde{A},\widetilde{i}))=(\widetilde{A}^{(p)},\widetilde{i}^{(p)},\ker(\Vi))$, where $\Frob$ and $\Vi$ are the usual Frobenius and Verschiebung. It can be easily checked that $\gamma\circ\delta=w_p\gamma \circ \delta w_p=\id$ and $\gamma\circ\delta w_p=w_p \gamma \circ\delta=w_p$.

According to \cite[Theorem 1.16]{D-R}, $\gamma$ and $w_p\gamma$ are closed morphisms and their images are respectively the two irreducible components of $\tilde X_0(D,N)$. These irreducible components meet transversally at the supersingular points of $\tilde X_0(D,N/p)$. More precisely, the set of singular points of $\tilde X_0(D,N)$ is in one-to-one correspondence with supersingular points of $\tilde X_0(D,N/p)$.

Let $P=[A,i]\in\CM(R)$ be a Heegner point and let $\widetilde P=[\widetilde{A},\widetilde{i}]\in \tilde X_0(D,N)$ denote its specialization. Assume, in addition, that $p$ does not divide the conductor of $R$. According to the above description of the special fiber, the point $\widetilde P$ is singular if and only if $\delta(\widetilde{A},\widetilde{i})$ is supersingular, or, equivalently, if and only if $(\widetilde{A},\widetilde{i})$ is supersingular because $(\widetilde{A},\widetilde{i})$ and $\delta(\widetilde{A},\widetilde{i})$ are isogenous. By \cite[Theorem 3.2]{Vig}, the fact that $\CM(R)\neq\emptyset$ implies that $p$ is not inert in $K$. Since $A$ is the product of two elliptic curves with CM by $R$, we obtain the following result,
wich the reader should compare with Theorem \ref{redCMsing}.

\begin{proposition}
A Heegner point $P\in \CM(R)$ reduces to a singular point of $\tilde X_0(D,N)$ if and only if $p$ ramifies in $K$.
\end{proposition}
\begin{proof}
The point $P$ specializes to a singular point if and only if $(\widetilde{A},\widetilde{i})$ is supersingular. Since $p$ is not inert in $K$ and any elliptic curve with CM by $R$ has supersingular specialization if and only if $p$ is not split in $K$, we conclude that $p$ ramifies in $K$.
\end{proof}

\subsection{Heegner points and the smooth locus}

Let $\cO'\supset\cO$ be an Eichler order in $B$ of level $N/p$. Notice that $\cO'$ defines the Shimura curve  $X_0(D,N/p)$. Let $\CM(R)$ be a set of Heegner points that specialize to non-singular points in $\tilde X_0(D,N)$ and let $P\in\CM(R)$. The inclusion $\cO'\supset\cO$ defines one of the two degeneracy maps $d:X_0(D,N)\rightarrow X_0(D,N/p)$ as in Appendix A. By the identification of \eqref{phip}, $P$ corresponds to $\varphi_P\in\CM_{D,N}(R)$ and its image $d(P)$ corresponds to $\varphi_{d(P)}\in\CM_{D,N/p}(R')$, where $\varphi_P(R')=\varphi_P(K)\cap\cO'$ and $\varphi_{d(P)}:R'\hookrightarrow\cO'$ is the restriction of $\varphi_P$ to $R'$. Since $[\cO':\cO]=p$, the inclusion $R\subseteq R'$ has also $p$-power index. According to the fact that the conductor of $R$ is prime to $p$, we deduce that $R=R'$.

Hence, restricting the natural degeneracy maps $X_0(D,N)\rightrightarrows X_0(D,N/p)$ to $\CM(R)$, we obtain a map
\begin{equation}\label{Phi4}
\CM_{D,N}(R)\ra\CM_{D,N/p}(R)\sqcup\CM_{D,N/p}(R).
\end{equation}

Observe that we have the analogous situation to \S\ref{CMsmooth} and Theorem \ref{CMsmooththm}. We have a map $\CM_{D,N}(R)\ra\CM_{D,N/p}(R)\sqcup\CM_{D,N/p}(R)$, which is clearly $\Pic(R)\times W(D,N/p)$ equivariant and a bijection if $N$ is square free, with the property that the natural map $$\CM_{D,N/p}(R)\sqcup\CM_{D,N/p}(R)\ra\Pic(D,N/p)\sqcup\Pic(D,N/p)$$ gives the irreducible component where the point lies. Notice that there are two irreducible components and $\#\Pic(D,N/p)=1$, since $D$ is the reduced discriminant of an indefinite quaternion algebra.

\subsection{Heegner points and the singular locus}

Let $\cO'\supset\cO$ be as above, let $(A,i)$ be an abelian surface with QM by $\cO'$ and let $C$ be a $\Gamma_0(p)$-structure.
Given the triple $(A,i,C)$, write $P=[A,i,C]$ for the isomorphism class of $(A,i,C)$, often
regarded as a point on $X_0(D, N)$ by Appendix A.

Let $P=[A,i,C]\in\CM(R)$ be a Heegner point with singular specialization in $\tilde X_0(D,N)$. Then $[A,i]\in X_0(D,N/p)$ is the image of $P$ through the natural map $d:X_0(D,N)\ra X_0(D,N/p)$ given by $\cO'\supset\cO$. Using the same argumentation as in the above setting, we can deduce that $\End(A,i)=\End(A,i,C)=R$.

Let $[\widetilde A,\widetilde i,\widetilde C]\in \tilde X_0(D,N)$ be its specialization. Since it is supersingular, $\widetilde C=\ker(\Fr)$. Thanks to the fact that the $\Fr$ lies in the center of $\End(\widetilde A)$, we obtain that $\End(\widetilde A,\widetilde i,\widetilde C)=\End(\widetilde A,\widetilde i)$. Thus the embedding $\End(A,i,C)\hookrightarrow\End(\widetilde A,\widetilde i,\widetilde C)$ is optimal and it is identified with $\End(A,i)\hookrightarrow\End(\widetilde A,\widetilde i)$, which has been considered in \S\ref{ssgood}. In conclusion we obtain a map $\Phi:  \CM_{D,N}(R)  \longrightarrow  \CM_{Dp,N/p}(R)$ as in \S\ref{CMsing} and, since $p$ ramifies in $K$, a result analogous to Theorem \ref{CMsingthm}.

\begin{theorem}\label{DRsing}
The map
\begin{equation}\label{Phi5}
\Phi:  \CM_{D,N}(R)  \longrightarrow  \CM_{Dp,\frac{N}{p}}(R)
\end{equation}
is equivariant for the action of $W(D,N)$ and, up to sign, of $\Pic(R)$. More precisely:
\[\Phi([J]\ast\varphi)=[J]^{-1}\ast\Phi(\varphi),\quad\quad\Phi(w_m(\varphi))=w_m(\Phi(\varphi))\]
for all $m\parallel DN$, $[J]\in\Pic(R)$ and $\varphi:R\hookrightarrow\cO$ in $\CM_{D,N}(R)$. Moreover, it is bijective if $N$ is square free.
\end{theorem}

\section*{Appendix A: Moduli interpretations of Shimura curves}

In this appendix, we describe an interpretation of the moduli problem solved by the Shimura curve $X_0(D,N)$, which slightly differs from the one already considered in \S \ref{ModCM}. This moduli interpretation is also well-known to the experts, but we provide here some details because of the lack of suitable reference.

Let $\{\cO_N\}_{(N,D)=1}$ be a system of Eichler orders in $B$ such that each $\cO_N$ has level $N$ and $\cO_N\subseteq\cO_M$ for $M\mid N$. Let now $M\parallel N$, by what we mean a divisor $M$ of $N$ such that $(M, N/M)=1$. Since $\cO_{N}\subseteq\cO_{M}$, there is a natural map $\delta:X_0(D,N)\rightarrow X_0(D,M)$; composing with the Atkin-Lehner involution $w_{N/M}$, we obtain a second map $\delta\circ w_{N/M}:X_0(D,N)\rightarrow X_0(D,M)$ and the product of both yields an embedding $\jmath: X_0(D,N)\hookrightarrow X_0(D,M)\times X_0(D,M)$ (cf.\,\cite{M-G} for more details).

Note that the image by $\jmath$ of an abelian surface $(A,i)$ over a field $K$ with QM by $\cO_N$ is a pair $\big((A_0,i_0)/K, (A_0',i_0')/K\big)$ of abelian surfaces  with QM by $\cO_M$, related by an isogeny $\phi_{N/M}:(A_0,i_0)\ra(A_0',i_0')$ of degree $(N/M)^2$, compatible with the multiplication by $\cO_M$. Assume that either $\car(K)=0$ or $(N,\car(K))=1$, thus giving a pair $\big((A_0,i_0), (A_0',i_0')\big)$ is equivalent to giving the triple $(A_0,i_0,C_{N/M})$, where $C_{N/M}=\ker(\phi_{N/M})$ is a subgroup scheme of $A_0$ of rank $(N/M)^2$, stable by the action of $\cO_M$ and cyclic as an $\cO_M$-module. The group $C_{N/M}$ is what we call a $\Gamma_0(N/M)$-structure.

Let us explain now how to recover the pair $(A,i)$ from a triple $(A_0,i_0,C_{N/M})$ as above. Along the way, we shall also  relate the endomorphism algebra $\End(A_0,i_0,C_{N/M})$ of the triple to the endomorphism algebra $\End(A,i)$. The construction of the abelian surface $(A, i)$ will be such that the triple $(A_0,i_0,C_{N/M})$ is the image of $(A,i)$ by the map $\jmath$. This will establish an equivalence of the moduli functors under consideration, and will allow us to regard points on the Shimura curve $X_0(D,N)$ either as isomorphism classes of abelian surfaces $(A,i)$ with QM by $\cO_N$ or as isomorphism classes of triples $(A_0,i_0,C_{N/M})$ with QM by $\cO_M$ and a $\Gamma_0(N/M)$-structure, for any $M\parallel N$.

Since $C_{N/M}$ is cyclic as a $\cO_M$-module, $C_{N/M}=\cO_M P$ for some point $P\in A_0$. We define $$\Ann(P)=\{\beta\in\cO_M,\;s.t.\;\beta P=0\}.$$
It is clear that $\Ann(P)$ is a left ideal of $\cO_M$ of norm $N/M$, since $C_{N/M}\cong \cO_M/\Ann(P)$. Let $\alpha$ be its generator. Assume $C=C_{N/M}\cap \ker(i_0(\alpha))$ and let $A:=A_0/C$ be the quotient abelian surface, related with $A_0$ by the isogeny $\varphi:A_0\ra A$.

We identify $C=\{\beta P,\;s.t.\;\beta\in\cO_M,\;\alpha\beta \in\Ann(P)=\cO_M\alpha\}=\{\beta P,\;s.t.\;\beta \in\alpha^{-1}\cO_M\alpha\cap \cO_M\}$.  The order $\alpha^{-1}\cO_M\alpha\cap \cO_M$ is an Eichler order of discriminant $N$, hence $C=\{\beta P,\;s.t.\;\beta\in\cO_N\}$.

Since $\cO_N=\{\beta\in\cO_M,\;s.t.\;\beta C\subseteq C\}\subseteq \cO_M$, the embedding $i_0$ induces a monomorphism $i:\cO_N\hookrightarrow \End(A)$ such that $i(B)\cap\End(A)=i(\cO_N)$. Then we conclude that $(A,i)$ has QM by $\cO_N$. Recovering the moduli interpretation of $X_0(D,N)$, it can be checked that $\delta(A,i)=(A_0,i_0)$.

Let $\End(A_0,i_0,C_{N/M})$ be the subalgebra of endomorphisms $\psi\in\End(A_0,i_0)$, such that $\psi(C_{N/M})\subseteq C_{N/M}$. Identifying $\End(A_0',i_0')$ inside $\End^0(A_0,i_0)$ via $\phi_{N/M}$, we obtain that $\End(A_0,i_0,C_{N/M})=\End(A_0,i_0)\cap\End(A_0',i_0')$.
Let $\psi\in\End(A_0,i_0,C_{N/M})$. Since $\psi$ commutes with $i_0(\alpha)$, we have that $\psi(\ker(i_0(\alpha)))\subseteq\ker(i_0(\alpha))$. Moreover it fixes $C_{N/M}$ by definition, hence $\psi(C)\subseteq C$ and therefore each element of $\End(A_0,i_0,C_{N/M})$ induces an endomorphism of $\End(A,i)$. We have obtained a monomorphism $\End(A_0,i_0,C_{N/M})\hookrightarrow\End(A,i)$.

\begin{proposition}\label{ChangInt}
Let $(A,i)$ obtained from $(A_0,i_0,C_{M/N})$ by the above construction. Then $\End(A,i)\simeq\End(A_0,i_0,C_{N/M})$.
\end{proposition}
\begin{proof}
We have proved that $\End(A_0,i_0,C_{N/M})\hookrightarrow\End(A,i)$, if we check that $\End(A,i)\hookrightarrow\End(A_0,i_0)$ and $\End(A,i)\hookrightarrow\End(A_0',i_0')$, we will obtain the equality, since $\End(A_0,i_0,C_{N/M})=\End(A_0,i_0)\cap\End(A_0',i_0')$.

Due to the fact that $C\subseteq C_{N/M}$, the isogeny $\phi_{N/M}$ factors through $\varphi$. By the same reason $i_0(\alpha)$ also factors through $\varphi$, and we have the following commutative diagram:
\[
\xymatrix{
A_0\ar[d]_{\phi_{N/M}}\ar[rd]_{\varphi}\ar[r]^{i_0(\alpha)} & A_0\ar[rd]^{\varphi} & \\
A_0' & A\ar[l]^{\rho}\ar[u]_{\eta}\ar[r]^{i(\alpha)} & A
}
\]
We can suppose that $\alpha$ is a generator of the two-sided ideal of $\cO_N$ of norm $N/M$. It can be done since $\cO_N=\cO_M\cap\alpha^{-1}\cO_M\alpha$, thus $\alpha$ can differ from the generator of such ideal by an unit of $\cO_M^{\times}$, namely an isomorphism in $\End_\C(A_0)$.

The orientation $\cO_N\subseteq\cO_M$ induces an homomorphism, $\mu:\cO_N\ra\cO_N/\cO_M\alpha\cong \Z/\frac{N}{M}\Z$.
We consider the subgroup scheme $$C_1=\{P\in\ker(i(\alpha))\;s.t.\;i(\beta) P=\mu(\beta)P,\;\mbox{for all }\beta\in\cO_N\}.$$
\emph{Claim:} $C_1=\ker(\eta)$. Clearly $\ker(\eta)\subset\ker(i(\alpha))$. Moreover $\ker(\eta)=\varphi(\ker(i_0(\alpha)))$ are those points in $\ker(i(\alpha))$ annihilated by $i(\cO_M\alpha)$, thus they correspond to the eigenvectors with eigenvalues $\mu(\beta)$, for all $\beta\in\cO_N$.

We also have the homomorphism, $\mu\circ w_{N/M}:\cO_N\ra\cO_N/\alpha\cO_M\cong \Z/\frac{N}{M}\Z$.
Again, we consider the subgroup scheme
$$C_2=\{P\in\ker(i(\alpha))\;s.t.\;i(\beta) P=\mu\circ w_{N/M}(\beta)P,\;\mbox{for all }\beta\in\cO_N\}.$$
\emph{Claim:} $C_2=\ker(\rho)$. In this case, we have that $\ker(\rho)=\varphi(C_{N/M})$, where $C_{N/M}=\cO_MP$. Due to the fact that $\alpha\cO_M\subseteq\cO_N$ and $\ker(\varphi)=\cO_NP$, we obtain that $\ker(\rho)\subseteq\ker(i(\alpha))$. By the same reason, $\ker(\rho)$ is the subgroup of $\ker(i(\alpha))$ annihilated by $i(\alpha\cO_M)$, therefore $\ker(\rho)=C_2$ as stated.

Finally let $\gamma\in\End(A,i)$. Since it commutes with $i(\beta)$ for all $\beta\in\cO_N$, it is clear that $\gamma(C_2)\subseteq C_2$. Then the isogeny $\rho$ induces an embedding $\End(A,i)\subseteq\End(A_0',i_0')$. Furthermore we have that $i_0(\alpha)i_0(\overline\alpha)=N/M$, then $\widehat{\varphi}=\eta\circ i_0(\overline{\alpha})=i(\alpha)\circ\eta$. Since $\gamma$ commutes with $i(\overline\alpha)$ and $\gamma(C_1)\subseteq C_1$, we obtain that $\widehat\varphi$ induces an embedding $\End(A,i)\subseteq\End(A_0,i_0)$. We conclude that $\End(A,i)=\End(A_0,i_0)\cap\End(A_0',i_0')=\End(A_0,i_0,C_{N/M})$.

\end{proof}


\bibliographystyle{plain}

\bibliography{santi}

\end{document}